\input ./style/arxiv-general.cfg
\documentclass[aop,MSNbibl,dvips]{arximspdf}
\makeatletter
   \@ifpackageloaded{graphicx}{}{\usepackage{graphicx}}
\makeatother

\doi{10.1214/13-AOP874} 
\volume{43}
\issue{3}
\pubyear{2015}
\firstpage{1082}
\lastpage{1120}

\makeatletter
\renewcommand{\mid}{|}
\newcommand{\rrvert}{\vert}
\newcommand{\llvert}{\vert}
\newcommand{\SLE}{\mathrm{SLE}}
\newtheorem{corollary}{Corollary}[section]
\newtheorem{lemma}[corollary]{Lemma}
\newtheorem{proposition}[corollary]{Proposition}
\newtheorem{theorem}[corollary]{Theorem}
\newproclaim{remark}{Remark}
\newproclaim{definition}{Definition}
\newproclaim{example}{Example}
\newcommand{\dist}{\operatorname{dist}}
\newcommand{\diam}{\operatorname{diam}}
\newcommand{\hcap}{\operatorname{hcap}}
\newcommand{\crad}{\operatorname{crad}}
\newcommand{\hatc}{{\hat c}}
\newcommand{\area}{\operatorname{Area}}
\newcommand{\hatG}{{\widehat G}}
\makeatother

\begin{document}
\begin{frontmatter}

\title{Minkowski content and natural parameterization for the
Schramm--Loewner evolution}
\runtitle{Minkowski content of SLE}

\begin{aug}
\author[A]{\fnms{Gregory F.}~\snm{Lawler}\corref{}\thanksref{T1}\ead[label=e1]{lawler@math.uchicago.edu}}
\and
\author[A]{\fnms{Mohammad A.} \snm{Rezaei}\ead[label=e2]{rezaei@math.uchicago.edu}}
\runauthor{G.~F. Lawler and M.~A. Rezaei}
\affiliation{University of Chicago}
\address[A]{Department of Mathematics\\
University of Chicago\\
5734 University Avenue\\
Chicago, Illinois 60637-1546\\
USA\\
\printead{e1}\\
\phantom{E-mail: }\printead*{e2}} 
\end{aug}
\thankstext{T1}{Supported by NSF Grant DMS-09-07143.}

\received{\smonth{2} \syear{2013}}
\revised{\smonth{7} \syear{2013}}

%
\begin{abstract}
We prove the existence and nontriviality
of the $d$-dimensional 4
Min\-kowski content for the Schramm--Loewner evolution
($\SLE_\kappa$) with $\kappa<8$ and $d = 1 + \frac{\kappa}8$.
We show that this is a multiple of the natural parameterization.
\end{abstract}

%

\end{frontmatter}

\section{Introduction}

A number of measures on paths or clusters on two-dimen\-sional
lattices arising from critical statistical mechanical
models are believed to exhibit some kind of conformal
invariance in the
scaling limit.
Schramm \cite{Sch} introduced a one-parameter family of such processes,
now called the
(\textit{chordal}) \textit{Schramm--Loewner evolution with parameter} $\kappa$
($\SLE_\kappa$)
and showed that these give the only possible limits for conformally
invariant processes in simply connected domains satisfying a certain
``domain Markov property.'' He defined the process as a
probability
measure on curves
from $0$ to $\infty$ in ${\mathbb H}$ and then used conformal invariance
to define the process in other simply connected domains.

The definition of the process in ${\mathbb H}$ uses
the half-plane Loewner equation. Suppose $\gamma\dvtx (0,t] \rightarrow
\overline{{\mathbb H}}$ is a
curve
with $\gamma(0) = 0$, and let $\gamma_t = \gamma(0,t]$.
Let $H_t$ denote the unbounded component of ${\mathbb H}\setminus
\gamma_t$.
We assume that $\gamma$ is noncrossing in the sense that for all
$ s <t $, $ \gamma[s,\infty) \subset
\overline H_s$, and $\gamma[s,t]
\cap H_s$ is nonempty. Let
$g_t\dvtx  H_t \rightarrow{\mathbb H}$ be
the unique conformal transformation with
$g_t(z) - z = o(1)$ as $z \rightarrow\infty$. Then for
every $a >0$, there exists a reparameterization of the curve such
that the following holds:
\begin{itemize}
\item For $z \in{\mathbb H}$, the map $t \mapsto g_t(z)$ is a smooth
flow and satisfies
the Loewner differential equation
\[
\partial_t g_t(z) = \frac{a}{g_t(z) - U_t},\qquad
g_0(z) = z,
\]
where $U_t$ is a continuous function on ${\mathbb{R}}$. This equation
is valid
up to a time $T_z \in(0,\infty]$.
\end{itemize}
Under the reparameterization, the transformation $g_t$ satisfies
\[
g_t(z) = z + \frac{at}{z} + O\bigl(|z|^{-2}\bigr),\qquad z
\rightarrow\infty.
\]
We say that the curve is \textit{parameterized by}
(\textit{half-plane}) \textit{capacity}.
Schramm-defined chordal $\SLE_\kappa$ to be the solution to the Loewner equation
with $a = 2$ and $U_t$ a Brownian motion with variance parameter
$\kappa$.
An equivalent definition (up to a linear time change)
is to choose $U_t$ to be
a standard Brownian motion and $a = 2/\kappa$. It has been shown that
a number of discrete random models have $\SLE$ as the scaling limit provided
that the discrete models are parameterized using (half-plane)
capacity. Examples are loop-erased random walk for $\kappa=2$
\cite{LSW}, Ising interfaces for $\kappa=3$ \cite{Smir2},
harmonic explorer for $\kappa=4$ \cite{SS}, percolation interfaces on
the triangular lattice for $\kappa=6 $ \cite{Smir1}
and uniform spanning trees for $\kappa=8$
\cite{LSW}.

If $D$ is a simply connected domain with distinct boundary points
$w_1,w_2$ where $\partial D$ is nice in neighborhoods of $w_1,w_2$, then
chordal $\SLE_\kappa$ from $w_1$ to $w_2$ in $D$ is defined by taking
the conformal
image of $\SLE_\kappa$ in the upper half plane under a transformation
$F\dvtx  {\mathbb H}\rightarrow D$ with $F(0) = w_1, F(\infty) = w_2$. The
map $F$ is not
unique, but scale invariance of $\SLE$ in ${\mathbb H}$ shows that the
distribution on
paths is independent of the choice. This can be considered as a measure
on the curves $F \circ\gamma$
with the induced parameterization or as a measure
on curves modulo reparameterization.

While the capacity parameterization is useful for analyzing the curve,
it is not
the scaling limit of the ``natural'' parameterization of the discrete
models. For
example, for loop-erased walks, it is natural to parameterize by the
length of
the random walk. One can ask whether the curves parameterized by a normalized
version of this ``natural length'' converge to $\SLE$ with a different
parameterization.
The Hausdorff dimension of the $\SLE$ paths \cite{Bf} is $d=1+\min\{\frac{\kappa}{8},1\}$.
It was conjectured in \cite{LS}
that the ``natural length'' of an $\SLE$ path might be given by the
$d$-dimensional Minkowski content
defined as follows.
Let
\[
{\operatorname{Cont}_d}(\gamma_t;r) = e^{r(2-d)} {\area} \bigl\{z\dvtx  \dist(z,\gamma_t) \leq e^{-r} \bigr\}.
\]
Then the
$d$-dimensional content
is
\[
{\operatorname{Cont}_d}(\gamma_t) = \lim
_{r \rightarrow\infty} {\operatorname{Cont}_d}(
\gamma_t;r),
\]
provided that the limit exists. If $\kappa\geq8$, then $d=2$, and the
two-dimensional Minkowski content is the same as the area and the
limit clearly exists. If $\kappa< 8$, it is not at all obvious that the
limit exists and is positive for $t > 0$. The main goal of this paper
is to prove this.

Before stating the
theorem, we will set some notational conventions for this paper. Let
$ 0 < \kappa< 8$ and let
\[
a = \frac{2} \kappa\in(1/4,\infty),\qquad d = 1 +\frac\kappa8 = 1 +
\frac{1}{4a} \in(1,2).
\]
Recall that $\gamma_t = \gamma(0,t]$ and we write $\gamma=
\gamma_\infty= \gamma(0,\infty)$ for the entire path
of the curve.
The Green's function for $\kappa< 8$
is defined by
\[
G(z) = \lim_{r \rightarrow\infty} e^{r(2-d)} {\mathbb P}\bigl\{\dist(z,
\gamma) \leq e^{-r} \bigr\}.
\]
This limit exists (see Section~\ref{greensec})
and there exists $c = c_\kappa$ such that
\[
G(z) = c [\Im z]^{d-2} [\sin\arg z]^{4a-1}.
\]
Our definition of the Green's function differs by a multiplicative
constant from that in other papers. If $F\dvtx  D\rightarrow
{\mathbb H}$ is a conformal transformation with $F(w_1) = 0$,
$F(w_2) = \infty$, we define
\[
G_{D}(z;w_1,w_2) = \bigl|F'(z)\bigr|^{2-d}
G\bigl(F(z)\bigr).
\]
There is also a two-point Green's function (see Section~\ref{twopointsec})
\[
G(z,w) = \lim_{r \rightarrow
\infty} e^{2r(2-d)} {\mathbb P}\bigl\{
\dist(z,\gamma) \leq e^{-r}, \dist(w,\gamma) \leq e^{-r} \bigr
\}.
\]

If $D \subset{\mathbb H}$, let
\[
G(D) = \int_D G(z) \,dA(z),\qquad G^2(D) = \int
_D \int_D G(z,w) \,dA(z) \,dA(w),
\]
where $dA$ denotes integration with respect to area.
We call $\gamma(t)$ a \textit{double point} for the $\SLE_\kappa$
path if there exists $s < t$ such that $\gamma(t) \in\partial H_s$.
If $0 < \kappa\leq4$, the $\SLE$ path has no double points
while they exist for $4 < \kappa< 8$.

\begin{theorem} \label{bigtheorem}
If $ 0 < \kappa<8$ there exists $\beta> 0$,
such that if $\gamma(t)$ is an
$\SLE_\kappa$ curve from $0$ to $\infty$ in ${\mathbb H}$ parameterized
by capacity, then with probability one, the following
holds:
\begin{itemize}
\item
For every $t > 0$, the Minkowski
content
\[
\Theta_t= {\operatorname{Cont}_d}(\gamma_t)
= \lim_{r \rightarrow\infty} {\operatorname{Cont}_d}(
\gamma_t;r),
\]
exists.
\item The function $t \mapsto\Theta_t$ is
strictly increasing and if $s < t$,
\[
\Theta_t - \Theta_s = {\operatorname{Cont}_d}
\bigl(\gamma[s,t]\bigr) = {\operatorname{Cont}_d}\bigl(\gamma(s,t] \cap
H_s\bigr).
\]
%
%
\item On every bounded interval $[0,t_0]$, $\Theta_t$
is H\"older continuous of order $\beta$.
\end{itemize}
Moreover,
if $D \subset{\mathbb H}$ is a bounded domain with piecewise
smooth boundary, then
\begin{eqnarray*}
{\mathbb E}\bigl[{\operatorname{Cont}_d}(\gamma\cap D)\bigr] &=& G(D),
\\
{\mathbb E} \bigl[{\operatorname{Cont}_d}(\gamma\cap
D)^2 \bigr] &=& G^2(D),
\end{eqnarray*}
and if $t > 0$,
\[
{\mathbb E}\bigl[{\operatorname{Cont}_d}(\gamma\cap D) \mid
\gamma_t\bigr] = {\operatorname{Cont}_d}(
\gamma_t \cap D) + \int_D G_{H_t}
\bigl(z;\gamma(t),\infty\bigr) \,dA(z).
\]
\end{theorem}

The proof will show that we can choose any $\beta<\alpha_0d/2$
where
\[
\beta< \frac{d}{2} \min\biggl\{ 1-\frac{\kappa}{24+2\kappa-8\sqrt
{8+\kappa}}, \frac{1}2
\biggr\} >0.
\]

The theorem allows us to define $\SLE_\kappa$ with the
\textit{natural parameterization} by letting
\[
\tilde\gamma(t) = \gamma(\sigma_t),\qquad \sigma_t = \inf\{s\dvtx
\Theta_s = t \}.
\]
Under this parameterization with probability one for
all $t$,
\[
{\operatorname{Cont}_d}(\tilde\gamma_t) = t.
\]

If $F\dvtx  {\mathbb H}\rightarrow D$ with $F(0) = w_1,F(\infty)
= w_2$ is a conformal transformation, then
as in \cite{LR} the natural parameterization
in $D$ can be defined by saying that the
time to traverse $F(\tilde\gamma[s,t])$ is
%
\begin{equation}
\label{aug291} \int_s^t \bigl|F'\bigl(
\tilde\gamma(r)\bigr)\bigr|^d\,dr.
\end{equation}
If $\tilde\gamma[s,t] \subset{\mathbb H}$, we can
see that this is the same as
$ {\operatorname{Cont}_d}[F \circ\tilde\gamma[s,t]] $. We expect
this to be true for all nice $D$. The only question is
the intersection of the curve
with the boundary for $4 < \kappa< 8$ with $D$ having
a nonsmooth boundary, perhaps of large dimension.

As
an example, let $D$ be the unit disk ${\mathbb D}$ and let $w_1 = 1,
w_2 = -1 $. In this case, the map $F\dvtx  {\mathbb H}\rightarrow{\mathbb
D}$ extends
analytically to ${\mathbb{R}}$ and there is no problem establishing
that (\ref{aug291})
equals $ {\operatorname{Cont}_d}[F \circ\tilde\gamma[s,t]] $.
Let $\gamma(t)$ be the $\SLE_\kappa$
path in ${\mathbb H}$ with the capacity parameterization,
and let $\eta(t) = F(\tilde\gamma(t))$ which is
an $\SLE_\kappa$ curve from $1$ to $-1$
in~${\mathbb D}$. Let
$ \Theta_t = {\operatorname{Cont}_d}[\eta_t]$.
In this case, $\Theta_\infty$ is an integrable random variable with
\[
{\mathbb E}[\Theta_\infty] = \int_{\mathbb D}G_{\mathbb D}(z;1,-1)
\,dA(z) < \infty.
\]
Moreover,
\[
{\mathbb E}[\Theta_\infty\mid\eta_t] = \Theta_t
+ \Psi_t,
\]
where
\[
\Psi_t = \int_{D_t} G_{D_t}\bigl(z;
\eta(t),-1\bigr) \,dA(z).
\]
Since $M_t:= {\mathbb E}[\Theta_\infty\mid\eta_t]$ is a
martingale, we
can see that $\Theta_t$ is the unique increasing process such that
$\Psi_t + \Theta_t$ is a martingale. This is a Doob--Meyer
decomposition.

In \cite{LS}, the natural parameterization was \textit{defined} to
be the unique process $\Theta_t$ which makes $\Psi_t +
\Theta_t$ a martingale. While this is a simple definition,
it requires moment bounds in order to make sure that the
process exists (uniqueness is easy). Indeed, it is not hard
to see that $M_t(z):= G_{D_t}(z;\eta(t), -1)$ is a local martingale,
and hence $\Psi_t$ is an integral of positive local martingales. If
$\Psi_t$ were also a local martingale, then no nontrivial $ \Theta_t$
could exist.

\begin{itemize}
\item In \cite{LS}, it was shown that for $\kappa< 5.0,\ldots,$ the process
$\Theta_t$ exists in ${\mathbb H}$ (the definition has to be modified slightly
in ${\mathbb H}$ because $\Psi_0$ as we have defined it above is
infinite---this is not very difficult). The necessary second moment bounds
were obtained using the reverse Loewner flow. It was shown that
for this range of $\kappa$, there exists $\alpha_0 = \alpha_0(\kappa)
>0$ such that the function $t \mapsto\Theta_t$
is H\"older continuous of order $\alpha$ for $\alpha<\alpha_0$.
\item In \cite{LZ}, the natural parameterization was shown to exist
for all $\kappa< 8$. There the necessary two-point estimates
were obtained from estimates on the two-point Green's function
\cite{Bf,LW}. However, the estimates were not strong enough to
determine H\"older continuity of the function $\Theta_t$.
\item In \cite{LR}, a new proof was given for all
$\kappa< 8$ combining ideas in \cite{LS,LZ} with known results
about the H\"older continuity of the Schramm--Loewner evolution
(with respect to the capacity parameterization). This established
continuity and H\"older continuity of the natural parameterization
for all~$\kappa$.
\end{itemize}

Let us discuss some conclusions that we can
derive. If $\Theta_t = {\operatorname{Cont}_d}(\gamma_t)$,
then clearly $\Theta_t$ is increasing and
measurable with respect to $\gamma_t$.
The conditional distribution of ${\operatorname{Cont}_d}[\gamma(t,
\infty)]$ given $\gamma_t$ is the same as the
distribution of the Minkowski content for $\SLE$
from $\gamma(t) $ to $-1$ in $D_t$. In
particular, using the fact that $\Theta_t - \Theta_s
= {\operatorname{Cont}_d}(\gamma(s,t] \cap D_s)$,
we have
\[
{\mathbb E}[\Theta_\infty- \Theta_t \mid
\gamma_t ] = \int_{D_t} \hatG_t(z) \,d
A(z),
\]
where $\hatG_t(z) = G_{D_t}(z;\gamma(t), -1)$.
Therefore,
\[
\Theta_t + \int_{D_t} \hatG_t(z) \,d
A(z)
\]
is a martingale.
Uniqueness of the Doob--Meyer decomposition shows that our
$\Theta_t$ must be the same as the natural parameterization
as discussed in \cite{LS,LZ,LR}. Using the Minkowski content
as the definition, we immediately\vadjust{\goodbreak} get independence of domain
as well as reversibility of the natural parameterization, that is,
the time to traverse $\gamma[s,t]$ is the same as the
time to traverse the path in the reverse direction. By
independence of domain, we mean that if $\gamma$ is an $\SLE_\kappa$
curve in ${\mathbb H}$ and
$D \subset{\mathbb H}$ with $\gamma(0,\infty) \subset D$,
then
the natural parameterization for $\gamma$
considered as an $\SLE$ curve in $D$ is the same as
that for the $\SLE$ curve in ${\mathbb H}$.
While this is clearly a property that we would expect
from a ``natural'' parameterization,
it is not at all obvious using the definition in \cite{LS}.

Another possible candidate for the ``natural length'' of an
$\SLE$ curve might be the $d$-dimensional Hausdorff measure.
However, it has been proved \cite{Rez} that this is zero with
probability one. It is unknown whether one can find a Hausdorff measure
with a different gauge function which would give a nontrivial quantity.

\subsection{Outline of the paper}

Section~\ref{prelimsec} sets notation for
the paper and reviews previous work.
We define the
Minkowski content in Section~\ref{minksec} and derive
some simple properties.
The Green's function for chordal $\SLE_\kappa$
is reviewed in Section~\ref{greensec}. This is a normalized
limit of the probability of getting near a point
$z$. We also discuss estimates in \cite{LW}
concerning the probability that an $\SLE_\kappa$ path
gets close to two points. In the following subsection,
we discuss some of the ideas used to prove two-point
estimates; in particular, some precise formulations are
made of the rough
statement ``after an SLE curve gets close to $z$ it is
unlikely to get close again.''
This section uses ideas from
\cite{Law3,LW}.

The proof of the main result is in the remainder
of the paper. Before going into specifics, let
us outline the basic idea of the proof. For ease,
let us fix a square, say $\Gamma= [0,1) +
i[1,2)$ and consider $\gamma\cap\Gamma$.
For each $z \in\Gamma$ and $r >0$, let
$\tau_r(z) = \inf\{t\dvtx  |\gamma(t) - z|
\leq e^{-r}\}$ and let $J_r(z)$ be
$e^{r(2-d)}$ times
the indicator function of the event $\{\tau_r(z)
< \infty\}$. Let $T_r(z)$ be the first
time that the conformal radius of $z$ in
${\mathbb H}\setminus\gamma(0,t]$ equals $e^{-r+2}$.
The Koebe $(1/4)$-theorem implies that
$T_r(z) < \tau_r(z)$. By comparison with
``two-sided radial'' ($\SLE$ conditioned to
go through $z$), one can show that
there exists $c_1$ such that
${\mathbb P}\{T_r(z) < \infty\} \sim c_1 G(z)
e^{r(d-2)}$. If $r$ is large, and we view
the path $\gamma[0,T_r(z)] $ near $z$, then
locally it appears like a path with the distribution
of two-sided radial $\SLE$. Using this, one can see
that
\[
{\mathbb P}\bigl\{J_r(z) >0 \mid T_r(z) < \infty
\bigr\} = \rho+o(1),\qquad r \rightarrow\infty,
\]
where $\rho$ is independent of $z$, and
using this in turn, we get a one-point estimate
%
\begin{equation}
\label{sep110} {\mathbb E}\bigl[J_r(z)\bigr] = c_1 \rho
G(z) + o(1).
\end{equation}
If we fix $\delta> 0$, we can see similarly
that there
exists $\rho'$ such that
\[
{\mathbb P}\bigl\{J_{r+\delta}(z) >0 \mid T_r(z) < \infty
\bigr\} = \rho' +o(1),\qquad r \rightarrow\infty,
\]
and by using (\ref{sep110}), we see that
$\rho' = e^{\delta(d-2)} \rho$. In other
words, ${\mathbb E}[J_{r+\delta}(z) - J_r(z)]
= o(1)$. The conditional
distribution of $ J_{r+\delta}(z)
- J_r(z)$ given $\gamma[0,T_r(z)]$ is determined
(up to a small error) by the way the curve $\gamma$
looks near $\gamma(T_r(z))
$, and this latter distribution
is understood through two-sided radial $\SLE_\kappa$.
If $z,w$ are not
very close together and the $\SLE$ curve gets close to
both $z$ and $w$, we might hope (and, indeed, this is
what we show) that the local behavior of $\gamma$
near $\gamma(T_r(z))$ and near $\gamma
(T_r(w))$ are almost
independent. The upshot of this is that if we
consider the random variable
\[
Y_r = \int_\Gamma\bigl[ J_{r+\delta}(z) -
J_r(z)\bigr] \,dA(z),
\]
then ${\mathbb E}[Y_r^2]$ is small. We show that
${\mathbb E}[Y_r^2] \leq c e^{-\beta r}$, from which
we conclude that
\[
\lim_{r \rightarrow\infty} \int_\Gamma
J_r(z) \,dA(z)
\]
exists as a limit in $L^2$ and with
probability one.

This outline is carried out in Section~\ref{mainsec}
assuming a moment bound, Theorem~\ref{keyestimate}
which is proved later. This establishes that with probability one
\mbox{${\operatorname{Cont}_d}[\gamma\cap\Gamma]$} exists for every
dyadic square $\Gamma$. Section~\ref{natsec}
uses this to prove the statements in Theorem
\ref{bigtheorem}, again leaving one fact for the
last section.

The main estimates are proved in the final
section. Section~\ref{onepointsec} analyzes
the one-point estimate, that is, the estimate
for getting close to a single point $z$. See Theorem~\ref{lime}. A key
to the two-point estimate is to understand
the one-point estimate very well. For ease,
we consider $\SLE$ in the disk between boundary
points and choose the origin to be the target
point. Two-sided radial, which is an example
of what are sometimes called $\SLE(\kappa,\rho)$
processes, describes chordal $\SLE$ ``conditioned
to go through $z$.'' It can be analyzed by a
one-dimensional SDE. We use this
to study $\SLE$ conditioned to get near $z$.
To do the two-point
estimate, we start in Section~\ref{twopointsec}
by reviewing the basic idea that after one
gets close to a point, one tends not to return
to it. This statement requires care to
make precise. See Lemmas~\ref{inpart} and~\ref{inpart2}. In the final
section, we complete
the proof giving a rigorous version of the
rough outline above.

\section{Preliminaries} \label{prelimsec}

\subsection{Notations and distortion}
\label{notsec}

We fix $\kappa< 8$ and allow all constants to depend implicitly
on $\kappa$. Recall that $a = 2/\kappa$ and $d = 1 + \frac{\kappa}{8}$.
If $\gamma$ is an $\SLE_\kappa$ curve from $w_1$ to $w_2$
in a simply connected domain $D$, we write
$\gamma_t = \gamma(0,t] = \{\gamma(s)\dvtx  0 < s \leq t \}$.

If $n,j,k$ are integers, we write $\Gamma_n(j,k) $
for the dyadic square
\[
\Gamma_n(j,k) = \bigl[j2^{-n}, (j+1)2^{-n}\bigr) \times i
\bigl[k2^{-n}, (k+1) 2^{-n}\bigr).
\]
Let
\begin{eqnarray*}
{\mathcal Q}_n &=& \bigl\{\Gamma_n(j,k)\dvtx  j \in{\mathbb
Z}, k \geq0\bigr\},\qquad
{\mathcal Q}_n^+ = \bigl\{\Gamma_n(j,k)
\in{\mathcal Q}_n\dvtx  k > 0\bigr\},
\\
{\mathcal Q}&=& \bigcup_{n \in{\mathbb Z}} {\mathcal
Q}_n,\qquad {\mathcal Q}^+ = \bigcup_{n \in{\mathbb Z}} {
\mathcal Q}_n^+.
\end{eqnarray*}
%
We
will need the
following simple distortion estimate.

\begin{lemma} \label{growth}
There exists $\delta> 0$ such that if $f\dvtx  {\mathbb D}\rightarrow
f({\mathbb D})$
is a conformal transformation with $f(0) = 0$, $|f'(0)| = \lambda$
and $|z| \leq\delta$
%
\[
|z| \exp\bigl\{-4|z|\bigr\} \leq\bigl|f^{-1}(\lambda z)\bigr| \leq|z| \exp\bigl\{4|z|\bigr\}.
\]
\end{lemma}

\begin{pf} By scaling, we may assume that $f'(0) = 1$.
The growth theorem (see, e.g., \cite{Law1}, Theorem 3.23),
states that for all $|z| < 1$,
\[
\frac{|z|}{(1+|z|)^2} \leq\bigl|f(z)\bigr| \leq\frac{|z|}{(1-|z|)^2}.
\]
Since $(1\pm|z|)^{-2} = 1 \mp2 |z| + O(|z|^2)$
and $\exp\{\pm4|z|\} = 1 \pm4|z| + O(|z|^2)$, we get the lemma.
\end{pf}

\subsection{Minkowski content} \label{minksec}

The $d$-dimensional Minkowski content is one way to ``measure'' the
size of a $d$-dimensional fractal. We use the quotes because the
content is not technically a measure. Its definition is in some ways
more natural
than $d$-dimensional Hausdorff measure; however, it has the
disadvantage that it is defined in terms of a limit that does not
always exist.
We will restrict our
consideration to $1 < d < 2$ and $V \subset{\mathbb C}$.

Let
\begin{eqnarray*}
{\operatorname{Cont}_d}(V;r) & = & e^{r(2-d)} \area\bigl\{z\dvtx
\dist(z,V) \leq e^{-r} \bigr\}
\\
& = & e^{r(2-d)} \int_{\mathbb C}1\bigl\{\dist(z,V) \leq
e^{-r} \bigr\} \,dA(z).
\end{eqnarray*}
Here, and throughout this paper, $dA$ denotes integration
with respect to two-dimensional Lebesgue measure.
The \textit{upper and lower $d$-dimensional Minkowski contents}
are defined by
\begin{eqnarray*}
{\operatorname{Cont}_d^+}(V;r) &=& \sup_{s \geq r} {
\operatorname{Cont}_d}(V;s),\qquad {\operatorname{Cont}_d^+}(V)
= \lim_{r \rightarrow\infty} {\operatorname{Cont}_d^+}(V;r),
\\
{\operatorname{Cont}_d^-}(V;r) &=& \inf_{s \geq r} {
\operatorname{Cont}_d}(V;s),\qquad {\operatorname{Cont}_d^-}(V)
= \lim_{r \rightarrow\infty} {\operatorname{Cont}_d^-}(V;r).
\end{eqnarray*}
The \textit{$d$-dimensional Minkowski content}
is defined if ${\operatorname{Cont}_d^+}(V) = {\operatorname
{Cont}_d^-}(V) $
in which case
\[
{\operatorname{Cont}_d}(V) = \lim_{r \rightarrow\infty} {
\operatorname{Cont}_d}(V;r).
\]

The following simple lemma lists the basic properties
of Minkowski content that we will use.

\begin{lemma} \label{symdifprop} $ $

\begin{itemize}
\item
If ${\operatorname{Cont}_d}(V), {\operatorname{Cont}_d}(V')$ exist
and $\dist(V,V') > 0$,
then ${\operatorname{Cont}_d}(V \cup V')$ exists and
\[
\label{aug250} {\operatorname{Cont}_d}\bigl(V \cup V'
\bigr) = {\operatorname{Cont}_d}(V) + {\operatorname{Cont}_d}
\bigl(V'\bigr).
\]
\item If ${\operatorname{Cont}_d}(V)$ exists, then
\[
{\operatorname{Cont}_d}(V) \leq{\operatorname{Cont}_d^+}
\bigl(V \cup V'\bigr) \leq{\operatorname{Cont}_d}(V) + {
\operatorname{Cont}_d^+}\bigl( V'\bigr).
\]
\item If $d > 1$ and
$D$ is a bounded domain whose boundary is a
piecewise analytic curve, then ${\operatorname{Cont}_d}(D) = 0$. If $V
\subset\overline D$, then
%
\begin{equation}
\label{jump} {\operatorname{Cont}_d}(V) = \lim
_{r \rightarrow\infty} e^{r(2-d)} \int_D 1\bigl\{
\dist(z,V) \leq e^{-r} \bigr\} \,dA(z),
\end{equation}
provided that either side exists.
\item
Suppose $V_1,V_2,\ldots$ are bounded
sets for which ${\operatorname{Cont}_d}(V_n)$ is well defined.
Let $V$ be a bounded set such that
\[
\lim_{n \rightarrow\infty} \bigl[{\operatorname{Cont}_d^+}(V
\setminus V_n) + {\operatorname{Cont}_d^+}(V_n
\setminus V) \bigr] = 0.
\]
Then ${\operatorname{Cont}_d}(V)$ exists and
%
\begin{equation}
\label{aug251} {\operatorname{Cont}_d}(V) = \lim
_{n \rightarrow\infty} {\operatorname{Cont}_d}(V_n).
\end{equation}
\end{itemize}
\end{lemma}

\begin{pf} We leave this to the reader. The last conclusion uses
\begin{eqnarray*}
{\operatorname{Cont}_d}(V_n;r) - {\operatorname{Cont}_d}(V_n
\setminus V;r) & \leq&{\operatorname{Cont}_d}(V;r)
\\
& \leq& {\operatorname{Cont}_d}(V_n;r) + {\operatorname
{Cont}_d}(V\setminus V_n;r).
\end{eqnarray*}\upqed
\end{pf}

\subsection{Green's function} \label{greensec}

The Green's function for chordal $\SLE_\kappa$ is the normalized
probability that the path gets near a point $z$. By nature,
it is defined up to a multiplicative constant and we choose the
constant in a way that will be convenient for us. The precise
definition uses the following theorem.
If $D$ is a simply connected
domain and $z \in D$ we let $\crad_D(z)$ denote the conformal
radius of $z$ in $D$, that is, if $f\dvtx {\mathbb D}\rightarrow D$ is
a conformal transformation with $f(0) = z$, then $\crad_D(z)
= |f'(0)|$.

\begin{theorem} \label{greentheorem}
For every $\kappa< 8$, there exists $ c' = c'(\kappa),
\hat c
= \hat c(\kappa), \alpha< \infty$ such that if $w = e^{2i\theta}
\in\partial{\mathbb D}$ and $\gamma$ is a chordal $\SLE_\kappa$
path from $1$ to $w$ in ${\mathbb D}$, then
%
\begin{eqnarray}\label{lemon}
{\mathbb P}\bigl\{\crad_{A}(0) \leq e^{-r}
\bigr\} &=& c' [\sin\theta]^{4a-1} e^{r(d-2)} \bigl[1 +
O\bigl(e^{-\alpha r}\bigr)\bigr],
\nonumber\\[-8pt]\\[-8pt]
{\mathbb P}\bigl\{\dist(0,\partial A) \leq e^{-r} \bigr\} &=& \hat c [
\sin\theta]^{4a-1} e^{r(d-2)} \bigl[1 + O\bigl(e^{-\alpha r}
\bigr)\bigr].\nonumber
\end{eqnarray}
Here, $A$ denotes the connected component
of ${\mathbb D}\setminus\gamma$
containing the origin.
\end{theorem}

\begin{pf} For the first expression see,
for example, \cite{LW}. The proof gives an
explicit form for $c'$ but we will not need it.
The second was proved in \cite{LR}, but
we reprove it here in Theorem~\ref{lime}.
This proof does not give an explicit expression
for the constant $\hat c$.
\end{pf}

To\vspace*{2pt} be precise, let
${\mathbb P}_\theta$ denote the probability distribution
on paths $\gamma= \gamma[0,\infty)$ given by
chordal $\SLE_\kappa$ from $1$ to $e^{2i\theta}$ in
${\mathbb D}$. Then there exist $c', \hat c, \alpha, c$, depending
only on $\kappa$,
such that for all $\theta$ and all $r \geq1/2$,
\begin{eqnarray*}
\bigl\llvert e^{r(2-d)} [\sin\theta]^{1-4a} {\mathbb
P}_\theta\bigl\{\crad_A(0) \leq e^{-r}\bigr\} -
c'\bigr\rrvert &\leq& c e^{-\alpha r},
\\
\bigl\llvert e^{r(2-d)} [\sin\theta]^{1-4a} {\mathbb
P}_\theta\bigl\{\dist(0,\gamma) \leq e^{-r}\bigr\} - \hat c
\bigr\rrvert &\leq& c e^{-\alpha r}.
\end{eqnarray*}
%
From the previously proven (\ref{lemon})
and the Koebe $(1/4)$-theorem,
we can easily deduce the following estimate which
we will use before deriving Theorem~\ref{lime}.
\begin{itemize}
\item If $\gamma$ is
an $\SLE_\kappa$ path from $0$ to $\infty$ in ${\mathbb H}$,
$\Im(z) \geq1$ and $r \geq0$,
%
\begin{equation}
\label{interior} {\mathbb P}\bigl\{\dist(z,\gamma) \leq e^{-r} \bigr\}
\asymp\bigl[\Im(z)\bigr]^{d-2} [\sin\arg z]^{4a-1}
e^{r(d-2)}.
\end{equation}
%
\end{itemize}
We also use the following estimate, see \cite{AK}.
\begin{itemize}
\item If $x > 0$,
%
\begin{equation}
\label{boundary} {\mathbb P}\bigl\{\dist(x,\gamma) \leq e^{-r} \bigr\}
\leq c \bigl[e^{-r}/x\bigr]^{4a-1}.
\end{equation}
\end{itemize}

If $w_1,w_2$ are distinct boundary points
of a simply connected domain $D$, let
$S_D(z;w_1,w_2)$ denote the sine of the argument of $z$ with
respect to $w_1,w_2$, that is, if $F\dvtx D \rightarrow{\mathbb H}$
is a conformal transformation with $F(w_1) = 0, F(w_2) = \infty$,
then $S_D(z;w_1,w_2) = \sin[\arg F(z)]$. Note that $S_D(z;w_1,w_2)$
is a conformal invariant and $\crad_D(z)$ is conformally covariant,
$\crad_{f(D)}(f(z)) = |f'(z)| \crad_{D}(z)$.
The \textit{chordal Green's function} is defined by
%
\begin{equation}
\label{greenfunction} G_D(z;w_1,w_2) = \hatc
\crad_D(z)^{d-2} S_D(z;w_1,w_2)^{4a-1}.
\end{equation}
Here, we choose the constant $\hat c$ from
Theorem~\ref{maintheorem}; our definition
differs from the definition
elsewhere (e.g., in \cite{LW})
by a multiplicative constant. Previously\vspace*{1pt} it
was defined so that $G_{\mathbb H}(z;0,\infty) = \Im(z)^{d-2}
[\sin\arg z]^{4a-1} = [\crad_{\mathbb H}(z)/2]^{d-2}\*
S_{\mathbb H}(z;0,\infty)^{4a-1}$.
The Green's function satisfies the conformal
covariance rule
\[
G_D(z;w_1,w_2) = \bigl|f'(z)\bigr|^{2-d}
G_{f(D)}\bigl(f(z);f(w_1), f(w_2)\bigr).
\]
We choose the definition (\ref{greenfunction}) so that we
do not need to keep writing the constant $\hat c$.
Theorem~\ref{greentheorem} extends immediately
to other simply connected
domains by conformal invariance of $\SLE$.

\begin{theorem} \label{Greenother}
If
$\kappa< 8$, $\gamma$ is a chordal $\SLE_\kappa$
path from $w_1$ to $w_2$
in a simply connected domain $D$, then
for $z \in D$ with $\dist(z,\partial D) \geq
2 e^{-r}$,
\[
{\mathbb P}\bigl\{\dist(z,\gamma) \leq e^{-r}\bigr\} =
G_D(z;w_1,w_2) e^{r(d-2)} \bigl[1 + O
\bigl(e^{-\alpha r}\bigr)\bigr],
\]
for some $\alpha>0$ which depends only on $\kappa$.
\end{theorem}

Most of our computations will be in the upper half plane
or in the disk.
For notational ease, we will write
\[
G(z) = G_{\mathbb H}(z;0,\infty),\qquad G(z;\theta) = G_{\mathbb D}
\bigl(z;1,e^{2\theta i}\bigr).
\]

If $V \subset{\mathbb H}$, we define
\[
G(V) = \int_V G(z) \,dA(z).
\]
Note that if $\Gamma\in{\mathcal Q}_n$ and $z$ is the center point
of $\Gamma$, then
\[
G(\Gamma) \asymp2^{-2n} G(z).
\]
[If $\Gamma= \Gamma_n(j,0)$, this requires a simple estimate of
an integral.]

\subsection{Two-point estimates}
\label{twopointsec}

A basic principle in proving two-point estimates for $\SLE$ is the
idea that if a path gets very close to a point $z$ and then gets away from~$z$, then it is
unlikely to get even closer to $z$. While
this is the heuristic, as just stated the principle is not
always valid. Since this idea is important in several of our proofs, we
will spend some time to formulate
and prove a precise version. We are expanding
on ideas in \cite{Law3,LR}. Let $\gamma$ be an $\SLE_\kappa$
curve from $0$ to $\infty$ in ${\mathbb H}$.
As before, if $z \in{\mathbb H}$, let
\[
\tau_r(z) = \inf\bigl\{t\dvtx  \bigl|\gamma(t) - z\bigr| = e^{-r} \bigr
\}.
\]

\begin{figure}[b]

\includegraphics{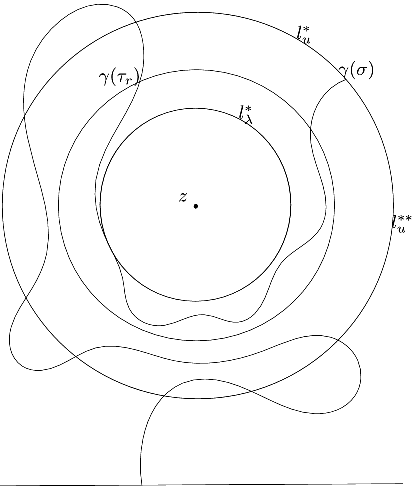}

\caption{Quantities in Section~\protect\ref{twopointsec}.}\label{simplefigure}
\end{figure}

If $\tau_{r}
= \tau_{r}(z) < \infty$,
let
$H =
H_{\tau_{r}}$ denote the unbounded
component of ${\mathbb H}\setminus\gamma_{\tau_{r}}$,
let ${\mathcal B}_u = {\mathcal B}_u(r,z)$ denote the disk
of radius $e^{-ur}$ containing $z$, and let ${\mathcal B}
= {\mathcal B}_1$ denote the disk of radius $e^{-r}$
about $z$.
Let $V_u = V_u(r,z)$
denote the connected component of ${\mathcal B}_u
\cap H$
containing $z$. If $u \leq1$,
the intersection of $\partial V_u $ with $
H $
is a disjoint
union of open arcs in $\partial{\mathcal B}_u$
each of whose endpoints is
in $\partial H$. There is a unique such arc $l$, that we denote
by $l_u = l_{u}(r,z)$, such that $z$ is contained in the bounded
component of $H \setminus l$. Simple connectedness
of $H$ is used to see that this arc is unique. However, one
may note the following facts:
\begin{itemize}
\item The bounded
component of $H \setminus l_{u}$ does not need to be
contained in ${\mathcal B}_u$. Indeed, we have no universal bound on the
diameter of the bounded component.
\item There may be other subarcs $l$ of $\partial{\mathcal B}_u \cap H$
such that $z$ is in the bounded component of $H \setminus l$.
However, these arcs are not on $\partial V_u$.
\end{itemize}
For $0 < u \leq1$, let
\[
\sigma=\sigma_u(r,z) = \inf\bigl\{t \geq\tau_{r}\dvtx
\gamma(t) \in\overline{l_{u}} \bigr\}.
\]
%
Then a correct, although still imprecise, version of our heuristic principle
is: if $\tau_{r} < \infty$, then
after time $\sigma$ the path is unlikely to get closer
to $z$. We will now be more precise.
Note that for fixed $z,u,r$, with probability one
$\gamma(\sigma) \in l_u$. In this case (which we now assume),
$l_u \setminus\{\gamma(\sigma)\}$ consists of two crosscuts of
$H_\sigma$ that we denote by $l^*_u$ and $l^{**}_u$. If $z
\in H_\sigma$, which is always true if $\kappa\leq4$, we
let $l^*_u$ be the crosscut such that $z$ lies in the bounded
component of $H_\sigma\setminus l^*_u$.
If $\tau_{r} < \infty$, define $\lambda= \lambda(r,z,u) \geq1 $ by
\[
\dist[z,\gamma_{\sigma}] = e^{-\lambda r}.
\]
Let $l_{\lambda}^*$ denote the connected subarc of
$\partial{\mathcal B}_\lambda\cap H_\sigma$ that separates $z$ from
infinity. (If the intersection of $\gamma_{\sigma}$ with
${\mathcal B}_\lambda$ is a single point, which we expect to be the
case with probability one, then $l_{\lambda}^*$ is a circle with a
single point
deleted.) See Figure~\ref{simplefigure} for a figure showing
illustrating these quantities.

Since we will use it in several proofs, we recall that if $D$ is
a domain and $\eta^1, \eta^2$ are disjoint subarcs
of $\partial D$, then the (Brownian) \textit{excursion measure} between
$\eta^1,\eta^2$ is given by
\[
{\mathcal E}_D\bigl(\eta^1,\eta^2\bigr) =
\int_{\eta^1} \partial_n \phi(z) |dz|,
\]
where $\partial_n$ denotes the inward normal derivative
and $\phi= \phi_{D,\eta^2}$ is the harmonic function on $D$ with
boundary value $1_{\eta^2}$. The above expression
assumes that $\eta^1$ is smooth, but one can check that
${\mathcal E}_D(\eta^1,\eta^2)$
is a conformal invariant and hence can be
defined for all domains. Also,
$ {\mathcal E}_D(\eta^1,\eta^2) = {\mathcal E}_D(\eta^2,\eta^1) $.
See \cite{Law1}, Chapter~5, for more details. When
estimating excursion measures, we will use the
following estimate that follows from the strong Markov
property.
Suppose $\eta$ is a crosscut of~$D$ that separates $\eta^1$
from $\eta^2$. Then
%
\begin{equation}
\label{newbie} {\mathcal E}_{D}\bigl(\eta^1,
\eta^2\bigr) \leq{\mathcal E}_{D \setminus\eta} \bigl(
\eta^1,\eta\bigr) \sup_{z \in\eta} \phi(z).
\end{equation}
If $D$ is simply connected, so that $(D,\eta^1,\eta^2)$ is
a conformal rectangle, and $ {\mathcal E}_D(\eta^1,\eta^2) \leq1$,
then
%
\begin{equation}
\label{newbie2} {\mathcal E}_D\bigl(\eta^1,
\eta^2\bigr) \asymp\max_{z \in D}
\phi_1(z) \phi_2(z),
\end{equation}
where $\phi_j = \phi_{D,\eta^j}$.
One can check this by verifying it for a rectangle
$[0,L] + i[0,\pi]$ by direct computation and using
conformal invariance.

\begin{lemma} \label{inpart}
There exists $c$ such that
for all $0 < u \leq1$,
\[
{\mathbb P}\bigl\{\dist[z,\gamma_\infty] < \dist[z,
\gamma_\sigma] \mid\gamma_\sigma\bigr\} \leq c
e^{\alpha(u-\lambda) r},
\]
where $\alpha= (4a-1)/2 >0$.
In particular,
%
\begin{equation}
\label{inparticular} {\mathbb P}\bigl\{\dist[z,\gamma_\infty] < \dist[z,
\gamma_\sigma] \mid\gamma_\sigma\bigr\} \leq c
e^{\alpha(u-1) r}.
\end{equation}
\end{lemma}

\begin{pf}
Let $g\dvtx H_\sigma\rightarrow{\mathbb H}$ be a conformal
transformation with $g(\gamma(\sigma)) = 0$, $g(\infty)
= \infty$.
The image $\eta= g \circ l^*_{u}$ is a crosscut of ${\mathbb H}$
with one endpoint on the origin and one on the real line
which without loss of generality we will assume is
on the positive real line. The
curve $\eta' = g \circ l^*_{\lambda}$ is a crosscut of ${\mathbb H}$ contained
in the bounded component of ${\mathbb H}\setminus\eta$
with positive endpoints $x_1 \leq x_2$.
Let us consider the conformal rectangle given by the component
of ${\mathbb H}\setminus(\eta\cup\eta')$ that contains
both $\eta$ and $\eta'$ on its boundary and with $\eta, \eta'$
as two of the boundary arcs of the rectangle. The excursion
measure
between $\eta$ and $\eta'$ in this rectangle is the
same as the excursion measure between $l^*_{u}$
and $l_{\lambda}^*$ for the corresponding rectangle in
$H_\sigma\setminus(l^*_{u} \cup
l_{\lambda}^*)$. The Beurling estimate (see, e.g., \cite{Law1}, Theorem 3.76)
implies that the latter is bounded
above by $c e^{-(\lambda- u)r/2}$. Since $\eta$ separates $\eta'$ from
the negative real line, we see that the excursion measure between
$\eta'$ and $(-\infty,0]$ in the unbounded component of
${\mathbb H}\setminus\eta'$ is bounded above by $
c e^{-(\lambda- u)r/2}$. By
standard estimates of the Poisson kernel in~${\mathbb H}$,
this shows that $\diam(\eta') \leq
c e^{-(\lambda- u)r/2} x_1$, and hence by (\ref{boundary}), the
probability that an $\SLE$ path hits it is $O(e^{-(\lambda- u)(4a-1)r/2})$.
\end{pf}

The next lemma strengthens (\ref{inparticular})
for $\kappa\leq4$. We do not know
if it is true for $4 < \kappa< 8$. Let ${\mathcal B}= {\mathcal B}_1$
denote the disk of radius $e^{-r}$ about $z$.

\begin{lemma} \label{inpart2}
If $\kappa\leq4$, there exists $c$ such that
if $0 < u \leq1$,
%
\begin{equation}
\label{inparticular2} {\mathbb P}\bigl\{ \gamma[\sigma,\infty) \cap
{\mathcal B} \neq
\varnothing\mid\gamma_\sigma\bigr\} \leq c e^{\alpha
(u-1)r
},
\end{equation}
where $\alpha= (4a-1)/2 >0$.
\end{lemma}

\begin{pf}
Let $V$ denote the unbounded component of
$H_\sigma\setminus\overline{\mathcal B}$ and note that
$l_{u}^*, l_u^{**} \subset V$. Let $L = \partial V \cap H_\sigma\cap
\partial
{\mathcal B}$ which is
a disjoint (finite or countable) union of open subarcs of $\partial
{\mathcal B}$,
which we denote by
$L_1,L_2,\ldots.$
For each arc $L_j$, either $l_u^*$ or $l_u^{**}$ disconnects
$L_j$ from infinity in $H_\sigma$, that is, $L_j$ is in
the bounded component of $H_\sigma\setminus l_u^*$ or
the bounded component of $H_\sigma\setminus l_u^{**}$.
Write $ L = L^1 \cup L^2$ where $L^1,L^2$ are
the unions of $L_j$ over the subarcs
of the first and second type, respectively.
The probability on the left-hand side
of (\ref{inparticular2}) is the probability that $\gamma[\sigma,
\infty) \cap L \neq\varnothing$. Hence, it suffices to show
that
\[
\sum_{j=1}^\infty{\mathbb P}\bigl\{ \gamma[
\sigma,\infty) \cap L_j \neq\varnothing\mid\gamma_\sigma
\bigr\} \leq c e^{\alpha
(u-1)r
}.
\]
We will give this bound for the sum over $L_j$ of
the first type; the sum over the second type is done
similarly.

Let ${{\mathcal R}}$ denote the bounded component
of $H_\sigma\setminus l_u^*$ which includes the $L_j$
of the first type. Using\vspace*{2pt} the Beurling estimate,
we can see that the excursion measure between $l_u^*$
and $L^1$ in $ {{\mathcal R}} \setminus L^1$,
$ {\mathcal E}_{{\mathcal R} \setminus L^1} (l_u^*, L^1)$, is
$O(e^{(u-1)r/2})$. We claim that a stronger fact is true,
%
\begin{equation}
\label{mar51} \sum_{j} {\mathcal
E}_{{\mathcal R} \setminus L_j}\bigl( l_u^*, L_j\bigr) \leq c
e^{(u-1)r/2},
\end{equation}
where we are summing over $L_j$ of the first type. Indeed,
$ {\mathcal E}_{{\mathcal R} \setminus L^1} (l_u^*, L^1)$
is the (integral over $l_u^*$ of the normal derivative of the)
probability that a Brownian motion starting at $z$ hits
$L^1$ before leaving $ {\mathcal R}$. The sum on the
left-hand side of (\ref{mar51}) is the (integral over $\ldots$ of the)
\textit{expected number}
of crosscuts $L^j$ visited before leaving~${\mathcal R}$.
However, using the strong Markov property and simple
connectedness, we can see that the probability starting
on one of the crosscuts $L_j$ of reaching another before leaving
${\mathcal R}$ is at most $1/2$, and hence the
expected number of crosscuts hit given one is hit is
at most $2$.

As in the previous proof, we use (\ref{boundary})
to see that
\[
\sum_l {\mathbb P}\bigl\{\gamma[\sigma,\infty)
\cap L_j \neq\varnothing\mid\gamma_\sigma\bigr\} \leq c \sum
_l {\mathcal E}_{{\mathcal R} \setminus L_j}\bigl(
l_u^*, L_j\bigr)^{4a-1}.
\]
The argument up to this point has not used the fact that
$\kappa\leq4$. However, if $\kappa\leq4$, we know
that $4a-1 \geq1$, and hence (\ref{mar51}) gives
\[
\sum_l {\mathcal E}_{{\mathcal R} \setminus L_j}\bigl(
l_u^*, L_j\bigr)^{4a-1} \leq\biggl[ \sum
_l {\mathcal E}_{{\mathcal R}
\setminus L_j}\bigl(
l_u^*, L_j\bigr) \biggr]^{4a-1} \leq c
e^{\alpha(u-1)r}.
\]\upqed
\end{pf}

While we do not know if the last lemma holds for $\kappa> 4$,
the next lemma will suffice for our needs.

\begin{lemma}
If $4 < \kappa< 8$, there exist $c<\infty, \beta> 0$ such that
if $\tau_{r} < \infty$ and $0 <
u \leq1$, then
\[
{\mathbb P}\{ {\mathcal B}\cap H_\sigma\neq\varnothing\mid
\gamma_{\tau_{r}} \} \leq c e^{\beta
(u-1)r}.
\]
\end{lemma}

\begin{pf} Let $\zeta= \gamma(\tau_{r})$.
Let $g$ be a conformal transformation of $H_{\tau_r}$
onto ${\mathbb H}$ with $g( \zeta) = 0, g(\infty) = \infty$.
Let $\eta= g \circ l_{u}, \eta' = g \circ[\partial{\mathcal B}
\setminus\{\zeta\}] $. Then
$\eta$ is a crosscut of ${\mathbb H}$ with one endpoint positive
and one endpoint negative, and $\eta'$ is a simple loop
rooted at the origin
lying in the bounded
component of ${\mathbb H}\setminus\eta$. By choosing a
multiple of $g$ if necessary, we may assume that
$\max\{|w|\dvtx  w \in\eta'\} = 1$.\vspace*{1pt}

We claim that there exists $c'$ such that
$\dist(0,\eta) \geq c' e^{(1-u)r/2}$.
To see this, let ${\mathcal R}$ denote the component of $H_\sigma
\setminus(\partial{\mathcal B}\cup l_u)$ whose boundary contains
both $\partial{\mathcal B} $ and $l_u$. Then using\vspace*{1pt}
the Beurling estimate
as in the previous lemma, we see that ${\mathcal E}_{{\mathcal R}}
(\partial{\mathcal B},l_u) \leq c e^{-(u-1)/2}$.
Therefore,
\[
{\mathcal E}_{{g(\mathcal R)}}\bigl(\eta,\eta'\bigr) = O\bigl(
e^{-(u-1)r/2}\bigr).
\]
But $\eta'$ is a connected set containing the origin of
radius $1$. If $v = \dist(0,\eta)$, then by setting
$z = i \sqrt v$ in (\ref{newbie2}) we get the bound
\[
{\mathcal E}_{{g(\mathcal R)}}\bigl(\eta,\eta'\bigr) \geq c
v^2.
\]

%

By conformal invariance,
$ {\mathbb P}\{ {\mathcal B}\cap H_\sigma\neq\varnothing\mid
\gamma_{\tau_{r}} \} $ is bounded above by the
probability that an $\SLE_\kappa$ path starting at
the origin has not separated the unit circle from
infinity before it reaches the circle of radius
$c' e^{(1-u)r/2}$. Using scaling and the fact that
$\SLE_\kappa$ has double points, it is not hard to show
that this is $O(e^{\beta(u-1) r})$ for some $\beta$.
\end{pf}

\begin{corollary} \label{hatrack}
If $\kappa< 8$, there exists $c<\infty$ and $\beta>0$ such that if
$|z| > e^{-r/2}$ and $0 < u \leq1$, then if $\tau_{r} < \infty$,
\begin{eqnarray*}
{\mathbb P}\bigl\{ \gamma[\sigma, \infty) \cap\overline{\mathcal B}\neq
\varnothing
\mid\gamma_{\tau_{r}} \bigr\} &\leq& c e^{\beta(u-1) r},
\\
{\mathbb P} \bigl\{\tau_r < \infty, \gamma[\sigma, \infty) \cap
\overline{\mathcal B}\neq\varnothing\bigr\} &\leq& c G(z) e^{(d-2)r}
e^{\beta(u-1) r}.
\end{eqnarray*}
\end{corollary}

The other estimates we need will deal with upper
bounds for the probabilities that the $\SLE$ curves
gets close to two different points $z,w$. For
the remainder of this section, we assume that $\gamma$
is an $\SLE$ curve from $0$ to $\infty$ in ${\mathbb H}$. If
$z \in\overline{\mathbb H}$, we let
\[
\tau_r(z) = \inf\bigl\{t\dvtx  \bigl|\gamma(t) - z\bigr| \leq e^{-r}
\bigr\}.
\]

\begin{lemma} There exists $ c < \infty$ such that
if
$|z|,|w| \geq e^{-u}$ and
$|z-w| \geq e^{-u}$, then for $0 < s < r$,
%
\begin{eqnarray}
\label{aug2050}
{\mathbb P}\bigl\{\tau_{s+u}(w) < \infty,
\tau_{r+u}(z) < \infty\bigr\} &\leq& c e^{(s+r)(d-2)},
\\
\label{aug205} {\mathbb P}\bigl\{\tau_{s+u}(w) < \tau_{r+u}(z)
< \tau_{r+u}(w) < \infty\bigr\} &\leq& c e^{2r(d-2)}
e^{-\alpha
s},
\end{eqnarray}
where $\alpha= (4a-1)/2$.
\end{lemma}

\begin{pf}
By scaling, it suffices to prove
the lemma for $u=0$. The first
estimate is Theorem 2 in \cite{LW}.
The second estimate follows from
the ideas in \cite{LW}, Lemma~4.10, but
we will redo the proof using some
ideas from this section.
Throughout this proof, we let $r,s,n$ be integers.

Let $\gamma= \gamma_{\tau_r(z)}$
and let $A = A_{s,r} $ denote the
$\gamma$-measurable
event
\[
A = \bigl\{ \tau_s(w) \leq\tau_r(z) <
\tau_{s+1}(w) \bigr\}.
\]
Let $n \geq s+1$ and let $E = E_{s,r,n} $ be the event
\[
E = \bigl\{ \tau_s(w) \leq\tau_r(z) \leq
\tau_n(w) < \tau_{r+1}(z) < \infty\bigr\}.
\]
The hard work is to show that on the event $A$,
%
\begin{equation}
\label{pat1} {\mathbb P}(E \mid\gamma) \leq c e^{-\alpha(r+s) }
e^{(d-2)(n-s)}.
\end{equation}
The estimate (\ref{aug2050}) shows that ${\mathbb P}(A_{s,r}) \leq O(e^{
(d-2) (r+s)})$ and the one-point estimate~(\ref{interior})
shows that
${\mathbb P}\{ \tau_n(z) < \infty\mid A \cap E\}
\leq O(e^{ (d-2)(n-r)})$. Hence, once
we establish (\ref{pat1}) we have
\[
{\mathbb P}\bigl\{ \tau_s(w) \leq\tau_r(z) \leq
\tau_n(w) < \tau_{r+1}(z) \leq\tau_{n}(z) <
\infty\bigr\} \leq c e^{-\alpha(r+s)} e^{2(d-2)n}.
\]
If we sum over $s$, we get
\[
{\mathbb P}\bigl\{ \tau_{r'}(z) \leq\tau_n(w) <
\tau_{r'+1}(z) \leq\tau_{n}(z) < \infty\bigr\} \leq c
e^{-\alpha r'} e^{2(d-2)n},
\]
and if we sum this over $r' \geq r$ we get
(\ref{aug205}).
We will prove (\ref{pat1}).
If $r+s \leq4$, we can estimate
\[
{\mathbb P}(E \mid\gamma) \leq{\mathbb P}\bigl\{\tau_n(w) < \infty
\mid\gamma\bigr\},
\]
and use the one-point estimate; hence we may assume that $r +s \geq4$.
We let
$s,r$ with $s + r \geq4$
and let $\tau= \tau_r(w)$.

Let $U^z$ (resp., $U^w$) denote the disk of radius $e^{-r/2}$
[$e^{-s/2}$] centered at $z$ [$w$]. Note that $U^z \cap U^w = \varnothing$.
For each $t \geq\tau$,
and $\zeta\in\{z,w\}$,
let $V_t^\zeta$ denote the connected component of
$H_t \cap U^\zeta$ that contains $\zeta$.
Let $\eta_t^\zeta$ denote the unique crosscut of
$H_t$ that is contained in $\partial V_t^\zeta\cap\partial U^\zeta$
and separates $z$ from $w$ in $H_t$. Let $l_t^\zeta$
denote the unique crosscut of $H_t$ contained in the circle of radius
$\dist(\zeta, \partial H_t)$ about $\zeta$ that separates
$z$ from $w$ in $H_t$. If there is a unique
point in $\partial H_t$ at minimal distance from $\zeta$,
then $l_t^\zeta$ is
a circle with one point removed. We will consider
four cases. Let $H = H_\tau, \eta= \eta_\tau^z$.
Let $\sigma$ be the fist time $t$ greater than or equal
to $\tau$ such that $z$ lies in the unbounded component
of $H_t \setminus\eta_t^z$.
\begin{longlist}[\textit{Case} 4:]
\item[\textit{Case} 1:]
Let $F_1 =A \cap\{\sigma= \tau\}$.
In this case, $\eta$ separates $w$ from $\gamma(\tau)$.
Using the Beurling estimate, we can see that the excursion measure
between $\eta$ and $l_\tau^w$ is $O(e^{-(r+s)/4})$; the latter
is a bound for
the probability that a Brownian motion starting on $l_{\tau}^w$
reaches $\eta$ without leaving $H$. The boundary estimate
(\ref{boundary})
states that the probability an $\SLE$ in $H$ starting at $\gamma(\tau)$
hits $l_{\tau}^w$ is $O(e^{-\alpha(r+s)})$.
Therefore, on the event $F_1$,
\[
{\mathbb P}\bigl\{\tau_{s+1}(w) < \infty\mid\gamma\bigr\} \leq c
e^{-\alpha(r+s)},
\]
and using the strong Markov property and the
one point estimate (\ref{interior}), we see that
\[
{\mathbb P}[E \cap F_1 \mid\gamma] \leq{\mathbb P}\bigl\{
\tau_{n}(w) < \infty\mid\gamma\bigr\} \leq c e^{-\alpha(r+s)}
e^{ (d-2)(n-s)}.
\]

\item[\textit{Case} 2:] Let $F_2 = A \cap\{\tau< \sigma< \tau_n(w)\}$.
We write
\[
F_2 = \bigcup_{j = s}^{n-1}
F_{s,j},
\]
where
\[
F_{2,j} = F_2 \cap\bigl\{\sigma_j(w) \leq
\sigma< \sigma_{j+1}(w) \bigr\}.
\]

Since the domain $H_t$ is decreasing,
in order for $z$ to change
from being in the bounded component of $H_{t'} \setminus
\eta_{s}^z,s' < t$ to being in the unbounded component of $H_t
\setminus
\eta_t^z$, the crosscut $\eta_t^z$ must be different from
$\eta^s_z$ for $s < t$. There
are two ways that the crosscut $\eta^z_t$ can change at
time $t$;
either $\gamma(t) \in\eta^z_{t-}$, or $\gamma(t) \notin\eta
^z_{t-}$ but $\eta^z_{t-}$ is not
part of the boundary of $V_t^z$. In the latter case, the
crosscut $\eta^z_{t-}$ still separates $z$ from
infinity and $b$ in $H_t$.
Also the crosscut $\eta^z_t$ separates $z$ from $\eta^z_{t-}$.
Hence, in the latter case
$z$ is in the bounded component of $H_{t} \setminus
\eta^z_t$.

Therefore, we see that $\gamma(\sigma) \in\eta_{\sigma-}^z$.
One
endpoint of the
crosscut $\eta_\sigma^z$ is $\gamma(\sigma)$ and it separates
$w$ from infinity. On the event $F_{2,j}$, the excursion measure
between $l^w_\sigma$ and $\eta_\sigma^z$ in $H_\sigma$ is bounded
above by $O(e^{-(r+j)/2})$. Therefore, on the event $F_{2,j}$,
\[
{\mathbb P}\bigl\{\tau_n(w) < \infty\mid\gamma_\sigma\bigr
\} \leq c e^{-\alpha(r+j)} e^{-(2-d)(n-j)}.
\]
The one-point estimate shows that
\[
{\mathbb P}[F_{2,j} \mid\gamma] \leq{\mathbb P}\bigl\{
\tau_j(w) < \infty\mid\gamma\bigr\} \leq c e^{-(2-d)(j-s)}.
\]
Therefore,
\[
{\mathbb P}[ E \cap F_{2,j} \mid\gamma] \leq c e^{-\alpha(r+j)}
e^{-(2-d)(n-s)},
\]
and by summing over $j=s,s+1,\ldots, n-1$, we see that
\[
{\mathbb P}[E \cap F_{2} \mid\gamma] \leq c e^{-\alpha(r+s)}
e^{-(2-d)(n-s)}.
\]

Before proceeding with the next cases, let $ \tau' = \tau_n(w),
H' = H_{\tau'}$,
and note that on $E \setminus(F_1
\cup F_2)$, we know that $z$ is in the bounded component of $H'
\setminus\eta_{\tau'}^z$.

\item[\textit{Case} 3:] Let $F_3$ be the intersection of $A \cap\{ \tau' <
\tau_{s+1}(z) \}$ with the event that $w$ is contained in the unbounded
component of $H' \setminus\eta_{\tau'}^w$. (Note\vspace*{1pt} that this is a
stronger condition
than saying that $z$ is contained in the bounded component
of $H'
\setminus\eta_{\tau'}^z$.) On the event $F_3$, the crosscut
$\eta_{\tau'}^w$ separates $l_{\tau'}^z$ from $\gamma(\tau')$ in
$H'$. The excursion measure between $l_{\tau'}^z$
and $\eta_{\tau'}^w$ in $H'$ is bounded by $O(e^{-(r+s)/2})$,
and using the boundary exponent, we see that on the event
$F_3$,
\[
{\mathbb P}\bigl\{\tau_{r+1}(z) < \infty\mid\tau'\bigr\}
\leq c e^{-(r+s) \alpha}.
\]
The one-point estimate implies that on $A$,
${\mathbb P}\{\tau' < \infty\mid
\gamma\} = O(e^{-(2-d)(n-s)})$, and hence
\[
{\mathbb P}[E \cap F_3 \mid\gamma] \leq c e^{-\alpha(r+s)}
e^{-(2-d)(n-s)}.
\]

\item[\textit{Case} 4:] Let $F_4$ be the intersection of
$[A \setminus(F_1 \cup F_2)] \cap
\{ \tau' <
\tau_{s+1}(z) \}$ with the event that $w$ is contained in the bounded
component of $H' \setminus\eta_{\tau'}^w$. As\vspace*{1pt} noted above,
on the event $F_4$, $z$ is in the bounded component of $H'
\setminus\eta_{\tau'}^z$. The excursion measure between
$l_{\tau'}^z$ and $\eta_{\tau'}^z$ in $H'$ is $O(e^{-r/2})$, and
as before this implies that on the event~$F_4$,
\[
{\mathbb P}\bigl\{\tau_{r+1}(z) < \infty\mid\gamma'\bigr
\} \leq c e^{-\alpha r},
\]
and hence on the event $A$,
%
\begin{equation}
\label{sept132} {\mathbb P}[E \cap F_4 \mid\gamma] \leq c
e^{-\alpha r} {\mathbb P}[F_4 \mid\gamma].
\end{equation}
Similar to case~2, let $\rho$ be the first time $t \geq
\tau_r(z)$ such that $w$ is
contained in the bounded component of $H_t \setminus
\eta^w_t$, and let
\[
F_{4,j} = F_4 \cap\bigl\{\tau_{j}(w) \leq
\rho< \tau_{j+1}(w) \bigr\}.
\]
Note that
%
\begin{equation}
\label{sept131} {\mathbb P}\bigl[\tau_{j}(w) \leq\rho<
\tau_{j+1}(w) \mid\gamma\bigr] \leq{\mathbb P}\bigl\{
\tau_j(w) < \infty\mid\gamma\bigr\} \leq c e^{(j-s)(d-2)}.
\end{equation}
As\vspace*{1pt} before, we can see that
the crosscut $\eta_\rho^w$ separates $l_\rho^w$ from $\gamma(\rho)$
in $H_\rho$. [Either $\rho= \tau_r(z)$ or
$\gamma(\rho)$ is an endpoint of $\eta_\rho^w$.]
Since the excursion measure between $\eta_\rho^w$
and $l_\rho^w$ in $H_\rho$ is $O(e^{-(j - ({s}/2))/2})$,
\[
{\mathbb P}\bigl[\tau_{j+1}(w) < \infty\mid\gamma_\rho
\bigr] \leq c e^{-(j-s) \alpha},
\]
and using the one point estimate,
\[
{\mathbb P}\bigl[\tau' < \infty\mid\gamma_\rho\bigr]
\leq c e^{-(s+j) \alpha} e^{(n-j)(d-2)}.
\]
Combining this with (\ref{sept131}) and summing over $
s \leq j\leq n$,
we see that
\[
{\mathbb P}[F_4 \mid\gamma] \leq c e^{-s\alpha}
e^{-(n-s)(2-d)}.
\]
Finally, combining this with (\ref{sept132}), we see that
\[
{\mathbb P}[E \cap F_4 \mid\gamma] \leq c e^{-\alpha r} {
\mathbb P}[F_4 \mid\gamma] \leq c e^{-(r+s)\alpha} e^{(n-s)(d-2)}.
\]%
\end{longlist}\upqed
\end{pf}

Given this estimate one also shows that
if $\Im(z),\Im(w) \geq1$ with $|z-w| \leq1$,
then
%
\begin{equation}
\label{sep120} {\mathbb P}\bigl\{\tau_{r}(z) < \infty,
\tau_r(w) < \infty\bigr\} \leq c e^{2r(d-2)}
|z-w|^{d-2}.
\end{equation}
Indeed, if $\rho= \inf\{t\dvtx  |\gamma(t) - z| \leq2 |z-w|\}$,
then
\[
{\mathbb P}\{\rho< \infty\} \leq c |z-w|^{2-d},
\]
and by conformal invariance,
\[
{\mathbb P}\bigl\{ \tau_{r}(z) < \infty, \tau_r(w) <
\infty\mid\rho< \infty\bigr\} \leq\bigl[e^{-r}/|z-w|
\bigr]^{2(2-d)}.
\]

In \cite{LW}, it was shown that
the limit,
\[
\lim_{\varepsilon.\delta
\downarrow0} \varepsilon^{d-2} \delta^{d-2} {
\mathbb P}\bigl\{\crad_{{\mathbb H}\setminus\gamma} (z_1) \leq\varepsilon,
\crad_{{\mathbb H}\setminus\gamma} (z_2) \leq\delta\bigr\},
\]
exists and defines a two-point Green's function.
In Section~\ref{onepointsec},
we show how to adapt this argument
to show existence of
%
\begin{equation}
\label{twogreen} G(z_1,z_2) = \lim_{\varepsilon.\delta
\downarrow0}
\varepsilon^{d-2} \delta^{d-2} {\mathbb P}\bigl\{
\dist(z_1,\gamma) \leq\varepsilon, \dist(z_2,\gamma) \leq
\delta\bigr\}.
\end{equation}
In fact, we can write $G(z_1,z_2) = \widehat G(z_1,z_2) +
\widehat G(z_2,z_1)$ where
\[
\widehat G(z,w) = G(z) {\mathbb E}^* \bigl[G_{H_T}(w;z,\infty) \bigr],
\]
and ${\mathbb E}^*$ denotes expectation with respect to two-sided
radial $\SLE_\kappa$ from $0$ to $z$ stopped at
\[
T = \inf\bigl\{t\dvtx  \gamma(t) = z\bigr\}.
\]
See Section~\ref{onepointsec} for a review of two-sided radial $\SLE$.

\section{Existence of Minkowski content}

\subsection{Main theorem} \label{mainsec}

If $\Gamma\in{\mathcal Q}_n$ as defined
in Section~\ref{notsec}, and $\gamma$
is an $\SLE_\kappa$ curve from $0$ to $\infty$
in ${\mathbb H}$, let
\begin{eqnarray*}
Z(\Gamma) &=& {\operatorname{Cont}_d^+}(\gamma\cap\Gamma;n \log2),
\\
J_r(z) &=& e^{(2-d)r} 1\bigl\{\tau_{r}(z) < \infty
\bigr\},\qquad J_r(V) = \int_V
J_r(z) \,dA(z).
\end{eqnarray*}
Note that if $s > 0$, then $J_{r+s}(z)
\leq e^{s(2-d)} J_r(z) $.

\begin{theorem} \label{maintheorem}
Suppose $\kappa< 8$ and $\gamma$ is
an $\SLE_\kappa$ curve from $0$ to $\infty$ in ${\mathbb H}$.
Then the following holds for
all $\Gamma\in{\mathcal Q}^+$.
\begin{itemize}
\item The limit
\[
\mu(\Gamma):= \lim_{r \rightarrow\infty} J_r(\Gamma)
\]
exists with probability one and in $L^2$.
\item With probability one,
%
\begin{equation}
\label{jan267} {\operatorname{Cont}_d}(\gamma\cap\Gamma) = \mu(
\Gamma).
\end{equation}

%
\item Let $\partial_n \Gamma= \{z \in{\mathbb H}\dvtx
\dist(z,\partial\Gamma) \leq2^{-n} \}$. Then with probability
one,
%
\begin{equation}
\label{aug293} \lim_{n \rightarrow\infty} {\operatorname{Cont}_d}(
\gamma\cap\partial_n \Gamma; n \log2) = 0.
\end{equation}
\item The following moment relations holds:
%
\begin{eqnarray}
\label{aug2210} {\mathbb E}\bigl[\mu(\Gamma)\bigr] &=& \int_\Gamma
G(z) \,dA(z),
\\
\label{aug2210more} {\mathbb E}\bigl[\mu(\Gamma)^2\bigr] &=& \int
_\Gamma\int_\Gamma G(z,w) \,dA(z) \,dA(w),
\\
\label{aug22100} {\mathbb E} \bigl[Z(\Gamma)^2 \bigr] &<& \infty.
\end{eqnarray}
\end{itemize}
\end{theorem}

Since ${\mathcal Q}^+$ is countable, all the with probability
one statements can be restated as ``with probability one, for
all $\Gamma\in{\mathcal Q}^+, \ldots.$''
The bulk of the work in proving the
theorem is to prove Theorem~\ref{keyestimate} below.
Let $0 < \delta<
1/10$.
Since we want to take a limit of $J_r$ as $r \rightarrow\infty$,
we will look at
\begin{eqnarray*}
Q_r^{ \delta} (z)& = & J_r (z) -
J_{r+\delta}(z)
\\
& = & {e^{r(2-d)} } \bigl[ 1 \bigl\{\tau_{r}(z) <\infty\bigr
\} - e^{\delta(2-d)} 1 \bigl\{\tau_{r+\delta}(z) <\infty\bigr\} \bigr].
\end{eqnarray*}
The random variable $Q_r^{ \delta} (z)$
is normalized so that $|Q_r^{\delta} (z)|$
is of order $1$ but ${\mathbb E}[Q_r^{\delta} (z)]$ is nearly zero.
The main estimate shows that if $r$ is large and
$z,w$ are not too close, then
$Q_r(z)$ and $Q_r(w)$ are almost independent.
%

\begin{theorem} \label{keyestimate}
Suppose $\kappa< 8$ and
$\gamma$ is $\SLE_\kappa$ from $0$
to $\infty$ in ${\mathbb H}$.
There exists $c < \infty, \beta> 0$
such that if
$0 < \delta< 1/10$,
$\Im(z),\Im(w) \geq
1$ and
$r \geq0$, then
%
\begin{equation}
\label{maines} {\mathbb E} \bigl[ Q_r^{ \delta} (z)
Q_r^{ \delta} (w) \bigr] \leq c e^{-\beta r}
|z-w|^{\beta-2}.
\end{equation}
%
\end{theorem}

We will not try to find the optimal $c,\beta$
in our proof.

\begin{pf*}{Proof of Theorem~\ref{maintheorem}
given Theorem~\ref{keyestimate}}
By scaling, we may assume that $\Gamma=[j,j+1) \times
i [k,k+1) \in{\mathcal Q}^+_0$ with $k \geq1$.
Suppose that $ 0 < \delta< 1/10$. Let $J_r =
J_r(\Gamma)$ and
\[
Q_r = Q_r^\delta= Q_r^{ \delta}
(\Gamma) = J_r - J_{r+\delta} = \int_\Gamma
Q_r^{ \delta} (z) \,dA(z).
\]
%
By
integrating (\ref{maines}), we see that
if $r \geq0$, then
$ {\mathbb E} [ Q_r ^2
] \leq c e^{-\beta r}$.
Let
\[
X_n = X_n^{ \delta} = J_0 + \sum
_{j=1}^n |Q_{j\delta}|.
\]
Then $X_n$ converges in $L^2$ to a random variable $X_\infty$. For each
positive integer
$n$, $|J_{n\delta} | \leq X_\infty$, and hence
%
\begin{equation}
\label{jan268} \sup_{r \geq0} J_r \leq
e^{\delta(2-d)} \sup_{n} J_{n\delta} \leq
e^{1/10} X_\infty.
\end{equation}
%
Also, if $n \leq m $,
\[
|J_{n\delta} - J_{m\delta}| \leq X_\infty- X_n.
\]
Therefore, $\{J_{n\delta} \}$ is a Cauchy
sequence in $L^2$ and has an $L^2$-limit which we call~$J_\infty$.
If $r\delta\leq s < (r+1)\delta$, we similarly have
%
\begin{equation}
\label{aug241} {\mathbb E}\bigl[(J_s - J_{r\delta}
)^2\bigr] = {\mathbb E}\bigl[\bigl(Q_{r\delta}^{
s-r\delta}
\bigr)^2\bigr] \leq c e^{-\beta r\delta},
\end{equation}
so we see that
\[
\lim_{s \rightarrow\infty} {\mathbb E} \bigl[(J_s -
J_\infty)^2 \bigr] \leq\lim_{s \rightarrow\infty} {
\mathbb E} \bigl[(J_s - J_{r\delta})^2 \bigr] + \lim
_{s \rightarrow\infty} {\mathbb E} \bigl[(J_{r\delta} -
J_\infty)^2 \bigr] = 0.
\]
Hence, $J_s \rightarrow J_\infty$ in $L^2$; in particular,
$J_\infty$
does not depend on $\delta$.

Chebyshev's inequality shows
that
\[
\sum_{n=1}^\infty{\mathbb P}\bigl\{|
Q_{n\delta}| \geq e^{-\beta n \delta/4} \bigr\} \leq\sum
_{n=1}^\infty\frac{{\mathbb E}[Q_{n\delta}^2]}{e^{-\beta n \delta/2 }
}< \infty.
\]
Hence, for each $\delta$, by the Borel--Cantelli lemma,
with probability one for all $n$ sufficiently large,
\[
\llvert J_{n\delta} - J_{(n+1)\delta} \rrvert\leq2 e^{-\beta n \delta/4}.
\]
This shows that with probability one, the sequence $\{J_{n\delta} \}$
is a Cauchy sequence, and hence with probability one,
for all $\delta= 2^{-m}$,
\[
\lim_{n \rightarrow\infty} J_{n\delta} = J_\infty.
\]
If $n\delta\leq r \leq(n+1)\delta$,
then
%
\begin{equation}
\label{aug201000} e^{\delta(d-2)} J_{(n+1) \delta} \leq J_r \leq
e^{\delta(2-d)} J_{n\delta},
\end{equation}
from which we conclude that with
probability one, for all $\delta= 2^{-m} $,
\[
e^{\delta(d-2)} J_\infty\leq\liminf_{r \rightarrow\infty}
J_r \leq\limsup_{r \rightarrow\infty} J_r \leq
e^{\delta(2-d)} J_\infty.
\]
Since this holds for all $\delta$, $J_r \rightarrow J_\infty$.

Note that for $r > 0$,
\begin{eqnarray*}
{\mathbb E} [J_r ] &=& \int_\Gamma
e^{(2-d) r} {\mathbb P}\bigl\{\tau_{r}(z) < \infty\bigr\} \,dA(z),
\\
{\mathbb E} \bigl[J_r ^2 \bigr] &=& \int
_\Gamma\int_\Gamma e^{2(2-d) r} {\mathbb
P}\bigl\{\tau_{r}(z), \tau_r(w) < \infty\bigr\} \,dA(z)
\,dA(w).
\end{eqnarray*}
Since $J_r \rightarrow J_\infty$ in $L^2$, we know that
\[
{\mathbb E}[J_\infty] = \lim_{r \rightarrow\infty} {\mathbb E}[J_r],\qquad
{\mathbb E}\bigl[J_\infty^2\bigr]= \lim_{r \rightarrow\infty} {\mathbb E}\bigl[J_r^2\bigr].
\]
Hence, (\ref{aug2210}) and (\ref{aug2210more})
follow from Theorem~\ref{greentheorem}
and (\ref{twogreen}). Indeed, the definition of the Green's
function (including the choice of multiplicative constant) was
made in order for these equalities to hold.

Note that if $ n \log2 \leq r \leq(n+1) \log2$,
\[
\bigl\llvert{\operatorname{Cont}_d}[\gamma\cap\Gamma; r] -
J_{r}(\Gamma)\bigr\rrvert\leq c J_{n \log2}(
\partial_n \Gamma).
\]
Using (\ref{interior}), we see that ${\mathbb E} [{\operatorname
{Cont}_d}(\gamma
\cap\partial_n \Gamma;
n \log2) ]
\leq c \area(\partial_n \Gamma) \leq c 2^{-n}$. Hence, using
the Markov inequality and the Borel--Cantelli lemma, we see
that with probability one for all $n$ sufficiently
large ${\operatorname{Cont}_d}(\gamma\cap\partial_n \Gamma;
n \log2) \leq2^{-n/2}$. This gives~(\ref{aug293}).
Also,
\begin{eqnarray*}
{\operatorname{Cont}_d}[\gamma\cap\Gamma; n \log2] & \leq&
J_{n\log2}
\\
& \leq& {\operatorname{Cont}_d}[\gamma\cap\Gamma; n \log2] + {
\operatorname{Cont}_d}(\gamma\cap\partial_n \Gamma; n
\log2).
\end{eqnarray*}
This gives
(\ref{jan267}).

Given $\Gamma\in{\mathcal Q}_0$, let $\Gamma_1,
\ldots, \Gamma_{12}$ denote the twelve squares
in ${\mathcal Q}_1$ whose interior does not intersect $\Gamma$
but whose boundary does. Note that these squares are in
${\mathcal Q}_1^+$. Any point within distance $1/2$ of
$\gamma\cap\Gamma$ is contained in
$\Gamma\cup\Gamma_1 \cup\cdots\cup\Gamma_{12}$,
and hence for
$r \geq1$,
\[
{\operatorname{Cont}_d}(\gamma\cap\Gamma; r) \leq J_{r}(
\Gamma) + J_r(\Gamma_1) + \cdots+ J_r(
\Gamma_{12}).
\]
This implies that
\[
{\operatorname{Cont}_d^+}(\gamma\cap\Gamma;0) \leq c + \sup
_{r \geq1} \bigl[J_{r}(\Gamma) + J_r(
\Gamma_1) + \cdots+ J_r(\Gamma_{12}) \bigr].
\]
Since $\Gamma_j \in{\mathcal Q}_1^+$,
the argument as in (\ref{jan268}), we see that for each $j$,
\[
\sup_{r \geq1} J_r (\Gamma_j)
\]
is square integrable. Hence, $ {\operatorname{Cont}_d^+}(\gamma\cap
\Gamma;0) $
is square integrable which\break gives~(\ref{aug22100}).
\end{pf*}

\subsection{Natural length} \label{natsec}

By Theorem~\ref{maintheorem}, with
probability one we can define
a function on ${\mathcal Q}$ by
\[
\mu(\Gamma) = {\operatorname{Cont}_d}(\gamma\cap\Gamma) = {
\operatorname{Cont}_d}\bigl(\gamma\cap{\operatorname{int}}(\Gamma)\bigr) = {
\operatorname{Cont}_d}(\gamma\cap\overline{\Gamma}).
\]

\begin{proposition}
On this event, $\mu$ extends to be a Borel measure.
\end{proposition}

\begin{pf} For each $\Gamma\in{\mathcal Q}^+$ and positive integer $r$,
we can define a Borel measure $\mu_r$ by stating
that the Radon--Nikodym derivative with respect to Lebesgue
measure is $J_r(z)$. Since $\mu_r(\Gamma) \rightarrow\mu(\Gamma)$
and $\Gamma$ is compact,
for each subsequence~$\{r_j\}$, there is a sub-subsequence
$\{r_{j_k}\}$ that converges to a measure $\mu'$ with
total mass $\mu(\Gamma)$. By using a diagonalization
argument, we can find a single sub-subsequence
such that the convergence holds
for all $\Gamma\in{\mathcal Q}^+$.
By (\ref{aug293}), we see that $\mu'(\partial\Gamma)
= 0$. Any open set $U$ can be written as a countable union
of squares $\Gamma\in{\mathcal Q}^+$ such that the interiors
of the squares are disjoint. Hence,
we can determine $\mu'(U)$ for any open set,
and hence we can see that $m' = \mu$ is unique.
\end{pf}

%

We call $\mu$ the (\textit{natural}) \textit{occupation measure} for
the $\SLE$ curve $\gamma$. If $D$ is an open set,
then we can find $D_n \in{\mathcal S}_{\mathbb H}$ increasing to $D$,
and hence
\[
{\mathbb E}\bigl[\mu(D)\bigr] = \int_D G(z) \,dA(z),\qquad {
\mathbb E}\bigl[\mu(D)^2\bigr] = \int_{D \times D} G(z,w)
\,dA(z) \,dA(w).
\]
It is not immediately obvious, but we
will now show that, with probability one, for all $0 \leq s < t <
\infty$,
%
\begin{equation}
\label{sep111} \mu\bigl(\gamma[s,t] \bigr) = \mu(\gamma_t \setminus
\gamma_s ) = {\operatorname{Cont}_d} \bigl( \gamma[s,t]
\bigr).
\end{equation}
The bulk of the work is in
the following
lemma. Recall that $H_s$ is the unbounded
component of ${\mathbb H}\setminus\gamma_s$, and let
\[
\partial_n H_s = \bigl\{z \in\overline H_s\dvtx
\dist(z,\partial H_s) \leq2^{-n}\bigr\}.
\]

\begin{lemma} \label{hardlemma}
There exists $\alpha> 0$ such that the following
holds with probability one.
\begin{itemize}
\item For each $t_0$, there exists $n_0 < \infty$ such that
if $0 \leq s \leq t_0$ and $n \geq n_0$, then
%
\begin{equation}
\label{aug301} {\operatorname{Cont}_d^+}\bigl(\gamma
\bigl[s,s+2^{-n}\bigr]\bigr) \leq2^{-n\alpha}.
\end{equation}
\item
Suppose that $0 \leq s < t $, and $u > 0$. Then
%
\begin{equation}
\label{aug302} \lim_{n \rightarrow\infty} {\operatorname{Cont}_d^+}
\bigl[ \gamma[s+u,t] \cap\partial_n H_s\bigr] = 0.
\end{equation}
\end{itemize}
\end{lemma}

The limit (\ref{aug302}) is immediate for $\kappa\leq4$ since
$\gamma[s+u,t] \cap\partial_n H_s$ is empty if $n$ is large.
Before proving the lemma, we will
show how to deduce
(\ref{sep111}) from the
lemma. We approximate $\gamma[s,t]
$ by intersections of $\gamma$ with
finite unions of dyadic squares.
If $s < t$ and
$n$ is a positive integer, let $V_n(s,t)$ denote the union of
all $\Gamma\in{\mathcal Q}_n$ satisfying $\Gamma\subset
H_s \setminus\partial_n H_s$ and
$ \gamma[s,t] \cap\Gamma\neq\varnothing$.
Let $O_n(s,t) = \gamma\cap V_n(s,t)$.
Note that $O_n(s,t) \subset\gamma\setminus\gamma_s$,
but it is possible for $\gamma(t,\infty) \cap O_n(s,t)
$ to be nonempty.
Note that if $u > 0$, then
\begin{eqnarray*}
O_n(s,t) \setminus\gamma[s,t] &\subset&\gamma[t,t+u] \cup\bigl(
\gamma[{t+u},\infty) \cap\partial_{n-1} H_t \bigr),
\\
\gamma[s,t] \setminus O_n(s,t) &\subset&\gamma[s,s+u] \cup\bigl[
\gamma({s+u},\infty) \cap\partial_{n-2} H_s \bigr].
\end{eqnarray*}
Here, we use the simple geometric facts that
$O_n(s,t) \cap\overline H_t \subset\partial_{n-1}H_t$,
and that if $\Gamma
\in{\mathcal Q}_n$, then either $\Gamma\subset
H_s \setminus\partial_n H_s
$ or $\Gamma\subset\partial_{n-2}
H_s$.
The lemma implies that
\[
\lim_{n \rightarrow\infty} {\operatorname{Cont}_d^+} \bigl[
\gamma[s,t] \setminus O_n(s,t) \bigr] + \lim_{n \rightarrow\infty}
{\operatorname{Cont}_d^+} \bigl[ O_n(s,t) \setminus
\gamma[s,t] \bigr] =0.
\]
Then (\ref{sep111}) follows from (\ref{aug251}).
The remainder of this subsection will be devoted
to proving the lemma. There is some technical work involved
here and are the basic reasons why the lemma holds.
\begin{itemize}
\item For (\ref{aug301}), we use the H\"older continuity
of an $\SLE$ path to say that that the diameter of
$\gamma[s,s+2^{-n}]$ is not very big. We also use moment
estimates to show that the Minkowski content is not
very big on any set of small diameter.
\item For (\ref{aug302}), we use the fact that
$
\gamma[s+u,t] \cap\partial_n H_s$ consists of points of the
curve that are either near the real line or are nearly double points
of the curve. We estimate moments for the content of
such paths.
\end{itemize}

We start by using the following lemma.

\begin{lemma} \label{aug24lemma1}
Let $Z(\Gamma)$ be defined
as before Theorem~\ref{maintheorem}.
There exists $c < \infty$ such that
if $\Gamma\in{\mathcal Q}_n^+$,
then
\[
{\mathbb E}\bigl[Z(\Gamma) \bigr] \leq c G(\Gamma),\qquad {\mathbb E}\bigl
[Z(\Gamma)
^2\bigr] \leq c 2^{-dn} G(\Gamma).
\]
\end{lemma}

\begin{pf} If\vspace*{1pt} $\widetilde\Gamma\in{\mathcal Q}_n^+$, then
$\Gamma= 2^n \widetilde\Gamma\in{\mathcal Q}_0^+$
with $G(\Gamma) = 2^{dn} G(\widetilde\Gamma)$.
Also, the distribution of $Z(\Gamma)$ is the same
as that of $2^{dn} Z(\widetilde\Gamma)$. Hence,
we may assume that
$\Gamma\in
{\mathcal Q}_0^+$.

If $\dist(0,\Gamma) \leq10$, then $G(\Gamma)
\asymp1$ and we can use
(\ref{aug22100}). Otherwise, let $\tau$ be the first
time that $\dist(\Gamma,\gamma(t)) = 8$. By (\ref{interior}),
${\mathbb P}\{\tau< \infty\} \asymp G(\Gamma)$, and
by distortion estimates we can see that
\[
{\mathbb E} \bigl[Z(\Gamma) \mid\tau< \infty\bigr] \leq c,\qquad {\mathbb E}
\bigl[Z(\Gamma)^2 \mid\tau< \infty\bigr] \leq c.
\]\upqed
\end{pf}

\begin{corollary} \label{aug27lemma1}
With probability one,
if $R < \infty$, then
for $n$ sufficiently large,
$\Gamma\in{\mathcal Q}_n$ with $\dist(0,\Gamma)
\leq R$,
\[
Z(\Gamma) \leq n 2^{-dn/2}.
\]
\end{corollary}

\begin{pf}
By Chebyshev's inequality,
if $\Gamma\in{\mathcal Q}_n$,
\[
{\mathbb P}\bigl\{Z(\Gamma) \geq n 2^{-nd/2}\bigr\} \leq n^{-2}
2^{nd} {\mathbb E}\bigl[Z(\Gamma)^2\bigr] \leq c
n^{-2} G(\Gamma).
\]
Hence, if $V$ is any bounded set,
\[
\sum_{n=0}^\infty\sum
_{\Gamma\in{\mathcal Q}_n,
\Gamma\subset V} {\mathbb P}\bigl\{Z(\Gamma) \geq n 2^{-nd/2}
\bigr\} \leq c \int_V G(z) \,dA(z) < \infty.
\]
The result follows from the Borel--Cantelli
lemma.
\end{pf}

Note that
\[
\partial_n {\mathbb H}= \bigl\{z \in{\mathbb H}\dvtx  \Im(z)
\leq2^{-n}\bigr\}.
\]

\begin{lemma} \label{aug27lemma2}
With probability one,
if $R < \infty$ and $u >0$, then for all $n$ sufficiently
large
\[
{\operatorname{Cont}_d^+} \bigl(\gamma\cap\partial_n {
\mathbb H} \cap\bigl\{|z| \leq R\bigr\} \bigr) \leq u 2^{-n}.
\]
In particular, for each $t_0$, for all $n$ sufficiently
large,
\[
{\operatorname{Cont}_d^+} \bigl(\gamma[0,t_0] \cap
\partial_n{\mathbb H} \bigr) \leq u 2^{-n}.
\]
\end{lemma}

\begin{pf}
The argument is the same
for all $R$; for ease, we let $R=1$
and write $V_n = \partial_n {\mathbb H}\cap\{|z|
\leq1\}$.
We will first
show that
%
\begin{equation}
\label{sep113} \sum_{n=1}^\infty{\mathbb P}
\bigl\{ {\operatorname{Cont}_d} (\gamma\cap V_n; n \log2)
\geq2^{-n} \bigr\} < \infty.
\end{equation}
Since $V_{n}
\subset\bigcup_{|j| \leq2^n}
\Gamma_n(j,0) $,
\begin{eqnarray*}
{\operatorname{Cont}_d}(\gamma\cap V_n;n \log2)& \leq&
\sum_{|j|
\leq2^n} {\operatorname{Cont}_d}\bigl(
\gamma\cap\Gamma_n(j,0);n \log2\bigr)
\\
& \leq& 6\cdot 2^{-2n} 2^{(2-d)n} \sum_{|j| \leq2^n}
1\bigl\{\gamma\cap\Gamma_n(j,0) \neq\varnothing\bigr\}.
\end{eqnarray*}
The estimate (\ref{boundary}) implies that
${\mathbb P}\{\gamma\cap\Gamma_n(j,0)
\neq\varnothing\} \leq c j^{1-4a} $.
If we choose $\beta$ with
$ 1 < \beta< d - (2-4a)_+$, we
can see that
\begin{eqnarray*}
{\mathbb E} \bigl[{\operatorname{Cont}_d}(\gamma\cap
V_n;n \log2) \bigr] &\leq& c 2^{-n \beta},
\\
{\mathbb P}\bigl\{{\operatorname{Cont}_d}(\gamma\cap
V_n; n \log2) \geq2^{-n}\bigr\} &\leq& c 2^{-n (\beta- 1)}.
\end{eqnarray*}
This gives (\ref{sep113}), and by the
Borel--Cantelli lemma with probability one for
all $n$ sufficiently large,
\[
{\operatorname{Cont}_d}(\gamma\cap V_n; n \log2)
\leq2^{-n}.
\]

It follows that for
$n$ sufficiently large, if $m \geq n$,
\[
{\operatorname{Cont}_d} (\gamma\cap V_n;m \log2 )
\leq2^{-m} + {\operatorname{Cont}_d} \bigl(\gamma\cap(
V_n \setminus V_m);m \log2 \bigr),
\]
and hence
\[
{\operatorname{Cont}_d^+} (\gamma\cap V_n )
\leq2^{2-d} \sup_{m \geq n} {\operatorname{Cont}_d}
\bigl(\gamma\cap( V_n \setminus V_m);m \log2 \bigr),
\]
where the supremum on the right is restricted
to integers $m$.
Let ${\mathcal A}_n$ denote the set of all
squares of the form
$\Gamma_l(j,1), l,j \in{\mathbb Z}$, that intersect $V_n$.
These squares are disjoint and
\[
{\operatorname{Cont}_d} \bigl(\gamma\cap( V_n \setminus
V_m);m \log2 \bigr) \leq\sum_{\Gamma\in{\mathcal A}_n} Z(
\Gamma).
\]
Hence,
\[
{\operatorname{Cont}_d^+} (\gamma\cap V_n )
\leq2^{2-d}\sum_{\Gamma\in{\mathcal A}_n} Z(\Gamma).
\]
By Lemma~\ref{aug24lemma1},
\[
\sum_{\Gamma\in{\mathcal A}_n} {\mathbb E} \bigl[Z(\Gamma) \bigr] \leq
c \sum_{\Gamma\in{\mathcal A}_n} G(\Gamma) \leq c G(V_n).
\]
As above, we find $\beta> 1$
such that
$ G(V_n)
\leq c 2^{-n \beta} $,
and hence
\begin{eqnarray*}
{\mathbb P} \biggl\{ 2^{2-d} \sum_{\Gamma\in
{\mathcal A}_n} Z(
\Gamma) \geq u 2^{-n } \biggr\} & \leq& u^{-1}
2^{n} {\mathbb E} \biggl[ 2^{2-d}\sum
_{ \Gamma\in{\mathcal A}_n} Z(\Gamma) \biggr]
\\
& \leq& c u^{-1} 2^{n(1-\beta)}.
\end{eqnarray*}
Hence, by the Borel--Cantelli lemma, with
probability one, for all $n$ sufficiently
large and all $m \geq n$,
\[
2^{2-d}\sum_{\Gamma\in{\mathcal A}_n} Z(\Gamma) \leq u
2^{-n}.
\]\upqed
\end{pf}

The next proposition establishes the H\"older continuity
of the function $t \mapsto{\operatorname{Cont}_d}(\gamma(0,t])$ and
completes the proof of (\ref{aug301}).

\begin{proposition} \label{holderprop}
There exists $\alpha> 0$
such that with probability one for every $t
< \infty$
for all $n$ sufficiently large and all $s \leq t$,
\[
{\operatorname{Cont}_d^+}\bigl(\gamma\bigl[s,s+ 2^{-n}
\bigr]\bigr) \leq2^{-n \alpha}.
\]
\end{proposition}

\begin{pf} It is known \cite{LJ1,Lind} that for
$\kappa\neq8$, the
$\SLE_\kappa$ curve is H\"older continuous
with respect to the capacity parameterization.
That is to say, there exists $\beta= \beta_\kappa> 0$ such
that with probability one, if $t < \infty$,
then for $n$ sufficiently large, and all $0 \leq s \leq t$,
\[
\diam\bigl(\gamma\bigl[s,s+2^{-n}\bigr] \bigr) \leq
2^{-n \beta}.
\]
Let $m$ be the largest integer less than $\beta n$.
Then $\gamma[s,s + 2^{-n}]$ is contained
in the union of
four rectangles $\Gamma_1,\ldots,
\Gamma_4 \in{\mathcal Q}_m$. For\vspace*{1pt} $n$ sufficiently
large, if $\Gamma_j \in{\mathcal Q}_m^+$, then
Corollary~\ref{aug27lemma1} implies that
${\operatorname{Cont}_d^+}(\Gamma_j \cap\gamma) \leq m 2^{-dm/2}$.
If $\Gamma_j \in{\mathcal Q}_n \setminus{\mathcal Q}_n^+$,
then Lemma~\ref{aug27lemma2} implies that
${\operatorname{Cont}_d^+}(\Gamma_j) \leq c 2^{-m} $.
The result follows for $\alpha< \beta d/2$.
\end{pf}

In the remainder of this section, we prove (\ref{aug302}) which is
\[
\lim_{n \rightarrow\infty} {\operatorname{Cont}_d^+}\bigl[
\gamma[s+u,t] \cap\partial_n H_s\bigr] = 0.
\]

Let $U = U_{j,k}
= \{x+iy\dvtx  -2^k \leq x < 2^k, y \geq2^{-j} \}$.
Using Lemma~\ref{aug27lemma2} and compactness
of $\gamma[t,u]$, we see that
it suffices to prove that with probability
one for every $s <u$ and all positive integers $j,k$,
%
\begin{equation}
\label{aug281} \lim_{n \rightarrow\infty} {\operatorname{Cont}_d^+}\bigl[
\gamma[u,\infty) \cap\partial_n H_s \cap
U_{j,k} \bigr] = 0.
\end{equation}
It suffices to consider rational $s,u$, and hence
we need to show that for fixed $s,u,j, k$, (\ref{aug281})
holds with probability one. By scaling, it
suffices to prove this for $j=0$ which we now assume. So we have
\[
U = U_{0,k} = \bigl\{x+iy\dvtx  -2^k \leq x <
2^k, y \geq1 \bigr\}.
\]
We fix integer $k >0$ and allow
constants to depend on $k$. We only consider
$n \geq k + 4$. Let $U = U_{0,k} $,
and let
${\mathcal Q}_n(U)$ denote the set of $\Gamma\in{\mathcal Q}_n$
with $\Gamma\subset U$. Note
that
$ G(\Gamma) \leq c 2^{-2n}$ if
$ \Gamma\in{\mathcal Q}_n(U) $.

We will now define a quantity $\widehat Z(\Gamma)$
for $\Gamma\in{\mathcal Q}$ that is an upper bound
for the Minkowski content of the intersection of
the path with $\Gamma$ ``after it has gotten close to the
square and then gotten away from
the square.''
To be precise, suppose that $\Gamma
\in{\mathcal Q}_n$ with center point $z$, and
define the following quantities:
\begin{itemize}
\item$\xi_1 =\xi_1(\Gamma) $ is the\vspace*{1pt} first time $t$ such
that
$|z-\gamma(t)| = 2^{-n+3}$. If $\xi_1 < \infty$,
let $l = l(\Gamma)$ denote a subarc
of the circle of radius $2^{-n/2}$ about $z$ such that
$z$ is in the bounded component of $H_{\xi_1} \setminus
l$. See Section~\ref{twopointsec} where a particular such arc
$l$ was selected. To be specific, we will make that choice
here.
\item$\xi_2 =\xi_2(\Gamma) $ is the first time
$t > \xi_1$
such that $\gamma(t) \in\bar l$.
\item$\xi_3 =\xi_3(\Gamma) $ is the first time $t > \xi_1$ such that
$|z-\gamma(t)| = 2^{-n+1}$.
\item$\xi_4= \xi_4(\Gamma)$ is the first time $t > \xi_2$
such that $|z-\gamma(t)| = 2^{-n+1}$.
\end{itemize}
We think of time $\xi_4$ as the time of the ``second return''
to the (neighborhood of the) square, see Figure~\ref{fig2}.

\begin{figure} 

\includegraphics{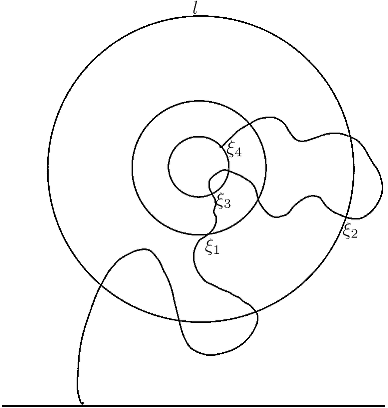}

\caption{The quantities
in Proposition~\protect\ref{holderprop} in the case $\xi_3 < \xi_2$.}\label{fig2}
\end{figure}

\begin{lemma} \label{lemma39}
There exists $n_0$ such that if $n \geq n_0$,
$\Gamma\in{\mathcal Q}_n$,
$\Gamma\cap U \neq\varnothing$, and $\xi_1 \leq s$,
then $\xi_2 < u$.
\end{lemma}

\begin{pf} The curve $\gamma$ is parameterized
so that $\hcap(\gamma[s_1,s_2] ) =
a(s_2 - s_1)$ where $\hcap$ is the
half-plane capacity which can be defined by
\[
\hcap(V) = \lim_{y \rightarrow\infty} y {\mathbb E}^{iy} \bigl[\Im
(B_\tau) \bigr],
\]
where $B_t$ is a standard Brownian motion and
$\tau= \tau_V = \inf\{t\dvtx  B_t \in{\mathbb{R}}\cup V\}$. In particular,
if $V_1 \subset V_2$,
\[
\hcap(V_2) - \hcap(V_1) \leq\lim_{y \rightarrow\infty}
y {\mathbb E}^{iy} \bigl[\Im(B_{\tau_2}); \tau_2 <
\tau_1 \bigr],\qquad \tau_j = \tau_{V_j}.
\]
Since the half-plane capacity is monotone,
\[
\hcap\bigl(\gamma[0,\xi_2] \bigr)\leq\hcap\bigl(\gamma[0,
\xi_1] \cup l \bigr).
\]
Using the Beurling estimate and the fact that $\Gamma\cap
U \neq\varnothing$, we can see that if $V_1 = \gamma(0,\xi_1],
V_2 = \gamma(0,\xi_1] \cap l$, then
\[
\lim_{y \rightarrow\infty} y {\mathbb E}^{iy} \bigl[
\Im(B_{\tau_2}); \tau_2 < \tau_1 \bigr] \leq\lim
_{y \rightarrow\infty} y {\mathbb P}^{iy} \{ \tau_2 <
\tau_1 \} \leq c \diam[l]^{1/2} \leq c 2^{-n/4}.
\]
If $n_0$ is chosen sufficiently large, then the right-hand side
is less than $u-s$ and hence $\xi_2 - \xi_1 < u-s$.
\end{pf}

We let
$E_1(\Gamma)$ be the event $\{\xi_1 < \xi_3 < \xi_2 <
\xi_4 < \infty\}$,
$E_2(\Gamma)$ the event
$\{\xi_1 < \xi_2 < \xi_3
= \xi_4
< \infty\}$, and
$E(\Gamma) = E_1(\Gamma) \cup E_2(\Gamma)
= \{ \xi_4< \infty\}$.
We define $\widehat Z(\Gamma)$
as follows:
\begin{eqnarray*}
\widehat Z(\Gamma) &=& 0\qquad\mbox{on the complement of } E(\Gamma),
\\
\widehat Z(\Gamma) &=& 2^{(n+1)(2-d)}\qquad\mbox{on the event } E_1(\Gamma),
\\
\widehat Z(\Gamma) &=& {\operatorname{Cont}_d^+}(\Gamma\cap\gamma; n
\log2)\qquad\mbox{on the event } E_2(\Gamma).
\end{eqnarray*}
Take $n_0$ as in the previous lemma and recall that
$U = U_{0,k}
= \{x+iy\dvtx  -2^k \leq x < 2^k, y \geq1 \}$. The definition is such that
the following holds:

\begin{itemize}
\item If $m > n_0$ and $\Gamma\in{\mathcal Q}_m$ with
$\Gamma\cap U \neq\varnothing$, then on the event
$E(\Gamma)$,
\[
{\operatorname{Cont}_d} \bigl[ \gamma[u,\infty) \cap\Gamma; m \log2 \bigr]
\leq\widehat Z(\Gamma).
\]
\item If $m \geq n >n_0$ and $\Gamma\in{\mathcal Q}_n$ with
$\Gamma\cap U \neq\varnothing$, then on the event
$E_2(\Gamma)$,
\[
{\operatorname{Cont}_d} \bigl[ \gamma[u,\infty) \cap\Gamma; m \log2 \bigr]
\leq\widehat Z(\Gamma).
\]
\end{itemize}

We will use the following fact which
states that once one gets close
to $z$ and then leaves, one is unlikely to return. It
is a quantitative
expression of the fact that the double points of $\SLE_\kappa$ curve
have strictly smaller fractal dimension than the curve itself. It is an
immediate
corollary of Corollary~\ref{hatrack}.

\begin{lemma} There exist $c,\beta$ such
that if $\Gamma\in{\mathcal Q}_n$
with $\Gamma\subset\{\Im(z) \geq1\}$,
then ${\mathbb P}[E(\Gamma)] \leq c 2^{n(d-2)} 2^{-n \beta}$.
\end{lemma}

Arguing as in the proof of Lemma~\ref{aug24lemma1},
we have on the event $E_2(\Gamma)$,
\[
{\mathbb E}\bigl[\widehat Z(\Gamma) \mid\gamma_{\xi_4}\bigr] \leq c
2^{n(2-d)}.
\]
Then we see that
\begin{eqnarray*}
{\mathbb E}\bigl[\widehat Z(\Gamma) \mid\xi_1( \Gamma) < \infty\bigr]
&\leq& c 2^{n(d-2)},
\\
{\mathbb E}\bigl[\widehat Z(\Gamma)\bigr] &\leq& {\mathbb P}\bigl\{\xi_1(
\Gamma) < \infty\bigr\} {\mathbb E}\bigl[\widehat Z(\Gamma) \mid\xi_1(
\Gamma) < \infty\bigr] \leq c G(\Gamma) 2^{-n \beta}.
\end{eqnarray*}
Let
\[
\widehat Z_n =\widehat Z_n(U) = \sum
_{m=n}^\infty\sum_{\Gamma
\in{\mathcal Q}_m, \Gamma\subset U}
\widehat Z(\Gamma).
\]
Then ${\mathbb E}[\widehat Z_n] \leq c 2^{-\beta n}$,
and hence using the Borel--Cantelli lemma, with
probability one for all $n$ sufficiently large, $\widehat Z_n \leq
2^{-\beta n/2} $. To establish (\ref{aug281}), it therefore suffices
to show
that there exists $c$ such that for $n > n_0$,
\[
{\operatorname{Cont}_d^+}\bigl[ \bigl(\gamma\setminus[
\gamma_u \cup\partial_nH_s] \bigr) \cap U
\bigr] \leq c \widehat Z_n.
\]
To show this it suffices to show for all integers $m \geq n$
\[
{\operatorname{Cont}_d}\bigl[ \bigl(\gamma\setminus[
\gamma_u \cup\partial_nH_s] \bigr) \cap U;m
\log2 \bigr] \leq c \widehat Z_n.
\]
For each $m \geq n > n_0$, will cover $\gamma\setminus[\gamma_u \cup
\partial_nH_s]$ by squares
$\Gamma\in{\mathcal Q}_j$ with \mbox{$n \leq j \leq m$}. We will choose all
squares $\Gamma\in{\mathcal Q}_m$ that intersect $H_s$
and are within distance $2^{-m+2}$ of $\partial H_s$. This includes
squares that intersect $\partial H_s$. However, for $n \leq j <m$, we
only choose squares whose distance from $\partial H_s$ is comparable
to~$2^{-j}$. In particular, these squares do not intersect $\partial H_s$.

To be precise, let $s < u $ and
assume that $n > n_0$. For fixed $ n < m$, let ${\mathcal A}
= {\mathcal A}_{s,m,n}$ denote the set of $\Gamma
\in{\mathcal Q}_j(U), j=n,\ldots,m $, that satisfy
$ 2^{-j+1} \leq\dist(\Gamma, \partial H_s)
\leq2^{-j+ 3 } $. Let ${\mathcal C} = {\mathcal C}_{s,m}$
denote the set of $\Gamma
\in{\mathcal Q}_m(U)$ that satisfy $\dist(\Gamma,
\partial H_s) \leq2^{-m+2}$. We claim that for each $m$,
the squares in
${\mathcal A} \cup{\mathcal C} $ cover $
U \cap\partial_n H_s$. To see this, suppose that $z \in U \cap
\partial_n H_s$. Then $\dist(z,\partial H_s) \leq2^{-n}$.
If $\dist(z, \partial H_s) \leq2^{-m+2}$, then the unique
$\Gamma\in{\mathcal Q}_m(U)$ containing $z$ is
in ${\mathcal C}$. If $\dist(z, \partial H_s) >
2^{-m+2}$ find
$j$ such that
$ 2^{-j+2} < \dist(z, \partial H_s) \leq2^{-j+3}$. Let
$\Gamma$ be the unique square in ${\mathcal Q}_j(U)$ that
contains $z$ and note that $2^{-j+1} \leq\dist(\Gamma, \partial_n H_s)
\leq2^{-j+3}$.

If
$\Gamma\in{\mathcal A} \cup{\mathcal C}$,
then $\xi_1 \leq s$. Since $n > n_0$, by Lemma~\ref{lemma39}
$\xi_2 < u$. Hence,
$\gamma[u,\infty) \cap\Gamma\subset
\gamma[\xi_4,\infty)$.
If $\Gamma\in{\mathcal A}$
with $\gamma[u,\infty) \cap\Gamma\neq\varnothing$,
then $\xi_2 < \xi_3$ which means that the
event $E_2(\Gamma)$ has occurred. If $\Gamma
\in{\mathcal C}$ and
$\gamma[u,\infty) \cap\Gamma\neq\varnothing$, we know
that $E(\Gamma)$ has occurred. Either way, we see that
\[
{\operatorname{Cont}_d}\bigl[\gamma[u,\infty) \cap\Gamma; m \log2 \bigr] \leq
\widehat Z_n (\Gamma).
\]
%

\section{Proof of Theorem \texorpdfstring{\protect\ref{keyestimate}}{3.2}}

Throughout this section, $0 < \delta\leq1/10$,
but constants are independent of $\delta$.

\subsection{Some reductions}

Suppose $\Im(z),\Im(w) \geq1$, and
let
$J_r(z), Q_r(z) = Q_r^{ \delta}
(z) $ as in Section~\ref{mainsec}.
Since $Q_r(z) Q_r(w) = 0$ if
$\tau_{r}(z) = \infty$ or $\tau_{r}(w) =
\infty$, in order to prove (\ref{maines}) it suffices by symmetry
to prove that
%
\begin{equation}
\label{mainesalt3} {\mathbb E} \bigl[Q_r(z) Q_r(w);
\tau_{r}(z) < \tau_{r}(w)< \infty\bigr] \leq c
e^{-\beta r} |z-w|^{\beta- 2}.
\end{equation}
By
(\ref{sep120}), we
know that if $|z-w| \leq e^{-ur}$,
\begin{eqnarray*}
{\mathbb E} \bigl[Q_r(z) Q_r(w) \bigr] & \leq& c
e^{2r(2-d)} {\mathbb P}\bigl\{\tau_r(z) < \infty,
\tau_r(w) < \infty\bigr\}
\\
& \leq& c |z-w|^{d-2}
\\
& \leq& c |z-w|^{(d-1) - 2} e^{-ur}.
\end{eqnarray*}
Hence, it suffices to find $u > 0$ and $c,\beta$ such that
(\ref{mainesalt3}) holds for $|z-w| \geq e^{-ur}$.
Let $\alpha$ be as in (\ref{aug205}), and
suppose that $s > 0$. Choose $u > 0$ with
$ u[2(2-d) + \alpha] \leq\alpha s/2$.

Then if $|z-w| \geq e^{-ur}$,
\begin{eqnarray*}
&& {\mathbb E} \bigl[Q_r(z) Q_r(w);
\tau_{sr}(w) \leq\tau_{r}(z) \leq\tau_{r}(w)<
\infty\bigr]
\\
&&\qquad \leq c e^{2r(2-d)} {\mathbb P}\bigl\{\tau_{sr}(w) \leq
\tau_{r}(z) \leq\tau_{r}(w)< \infty\bigr\}
\\
&&\qquad \leq c e^{2r(2-d)} {\mathbb P}\bigl\{\tau_{sr- ur + ur}(w)
\leq
\tau_{r
-ur + ur}(z) \leq\tau_{r-ur+ur}(w) < \infty\bigr\}
\\
&&\qquad \leq c e^{2r(2-d)} e^{2(r-ur)(d-2)} e^{-\alpha(sr-ur)}
\\
&&\qquad \leq c e^{ru[2(2-d) + \alpha]} e^{- \alpha sr} \leq c
e^{-\alpha s r/2}.
\end{eqnarray*}
From this, we see that in order to prove (\ref{mainesalt3}) it
suffices to prove the following.
There exist $u>0, s > 0, \beta> 0, c < \infty$
such that if $\Im(z), \Im(w) \geq1$ and $|z-w| \geq e^{-ur}$, then
%
\begin{equation}
\label{mainesalt} {\mathbb E} \bigl[Q_r(z) Q_r(w);
\tau_{r}(z) < \tau_{sr}(w) < \tau_{r}(w) <
\infty\bigr] \leq c e^{-\beta r} |z-w|^{\beta- 2}.
\end{equation}
This is what we will establish in this section.

\subsection{One-point estimate} \label{onepointsec}

As is often the case, an important step in getting a two-point
estimate is to get a sharp one-point estimate with good
control on the error terms. Much of the necessary analysis
has been done for $\SLE$, and we review
some of the methods here.

We will consider chordal $\SLE_\kappa$ from $1$ to
$w = e^{2i\theta}$ in the unit disk ${\mathbb D}$, and we will study
how close the path gets to the origin.
We parameterize the $\SLE_\kappa$ path
$\gamma$
using the radial parameterization. To be specific,
we let $D_t$ denote the component of ${\mathbb D}
\setminus\gamma_t$ containing the origin and $g_t\dvtx  D_t
\rightarrow{\mathbb D}$ the unique conformal transformation with
$g_t(0) = 0, g_t(\gamma(t)) = 1$. The radial parameterization
is defined so that
$|g_t'(0)| = e^t$. The total lifetime of the curve
in this parameterization, $T$, is
finite with probability one. (If $\kappa
> 4$, $T$ is not the time that the curve reaches $w$, but rather the
time at which the curve disconnects the origin
from $w$. Although the $\SLE$ curves continues after this time, the
domain $D_t$ does not change so we do not need to consider the path
after time $T$.) We write $w_t = e^{2i\theta_t}=
g_t(w)$. The path~$\gamma_t$, and hence the transformations
$g_t$, are determined by $\theta_s, 0 \leq s \leq t$.
If $t > T$, then $g_t = g_T$.
We let ${\mathbb P}_\theta,{\mathbb E}_\theta$ denote probabilities
and expectations given by chordal $\SLE_\kappa$ from $1$ to
$e^{2i\theta}$. The angle $\theta_t$
satisfies a simple one-dimensional SDE.
Its form is a little nicer if we consider a
linear time change. If $\hat\theta_t = \theta_{2at}$, then $\hat
\theta_t$
satisfies the ``radial Bessel equation''
\[
d\hat\theta_t = (1-2a) \cot\hat\theta_t \,dt + dB_t,
\]
where $B_t$ is a standard Brownian motion. This equation is valid
until the time $\widehat T = T/2a$ at which $\hat\theta_{\widehat T}
= \theta_T \in\{0,\pi\}$.

The Koebe $(1/4)$-theorem and Schwarz lemma implies
that for $0
\leq t \leq T$,
%
\begin{equation}
\label{koebe} e^{-t -\log4} \leq\dist[0,\gamma_t] \leq
e^{-t}.
\end{equation}
Let $S_t = S_{D_t}(0;w,\gamma(t)) = S_{{\mathbb D}}(0;w_t,1) =
\sin\theta_t $. It\^o's formula shows
that
\[
M_t = 1\{T > t\} e^{t(2-d)} S_t^{4a-1}
\]
is a local martingale; more precisely, $\widehat M_t
= M_{2at}$ satisfies
\[
d \widehat M_t = (4a-1) [\cot\hat\theta_t] \widehat
M_t \,dB_t,\qquad t < \widehat T.
\]
In fact, $M_t$ is a continuous martingale with
${\mathbb P}\{M_T = 0\} = 1$.

Let ${\mathbb D}_r$ denote the open disk of radius $e^{-r}$ about the
origin with closure $\overline{\mathbb D}_r$, and
\[
\tau_r = \inf\bigl\{t\dvtx  \dist[0,\gamma_t] =
e^{-r} \bigr\} = \inf\bigl\{t\dvtx  \gamma(t) \in\overline{\mathbb
D}_r\bigr\}.
\]
Note that (\ref{koebe}) implies that
\[
r - 2 < r - \log4 \leq\tau_r \leq r.
\]

The measure obtained by tilting by the martingale $M_t$ is
called \textit{two-sided radial $\SLE_\kappa$} (\textit{from $0$ to $e^{2i\theta}$
in ${\mathbb D}$
going through the origin stopped when it reaches
the origin}). We will
write ${\mathbb P}^*,{\mathbb E}^*$ for probabilities and expectations with
respect to this measure. These measures depend on the
initial angle $\theta$ and we will write ${\mathbb P}^*_\theta,{\mathbb
E}^*_\theta$
if we wish to make this explicit.
The quantity ${\mathbb E}^*_\theta$ is defined by saying that
if $X$ is a random variable
that depends
only on $\gamma_t$, then
\[
{\mathbb E}^*_\theta(X) = M_0^{-1} {\mathbb
E}_\theta[X M_t ] = [\sin\theta]^{1-4a}
e^{t(2-d)} {\mathbb E}_\theta\bigl[X S_t^{4a-1}
1\{T > t\} \bigr],
\]
or equivalently,
%
\begin{equation}
\label{aug84} {\mathbb E}_\theta\bigl[X 1\{T > t\} \bigr] =
e^{(d-2)t} [\sin\theta]^{4a-1} {\mathbb E}^*_\theta
\bigl[X S_t^{1-4a} \bigr].
\end{equation}

The Girsanov theorem shows that
under the measure ${\mathbb E}^*_\theta$,
%
\begin{equation}
\label{radbes} d \hat\theta_t = 2a \cot\hat\theta_t \,dt +
d W_t,
\end{equation}
where, as before, $\hat\theta_t = \theta_{2at}$
and $W_t$ is a standard Brownian motion with respect to
the tilted measure.
This equation has an invariant probability density
\[
\phi(\theta) = C_{4a}^{-1} [\sin\theta]^{4a},\qquad
C_{4a} = \int_0^{\pi}
\sin^{4a} \theta \,d\theta.
\]
Moreover, the rate of convergence to equilibrium
is exponential (see, e.g., \cite{LZ}, Section~2.1.1). To be
more explicit, there exists $\alpha> 0$ such
that if $\phi_t(\theta; \theta_0)$ is the density at
time $t$ given initial condition $\theta_0$, then
%
\begin{equation}
\label{aug91} \phi_t(\theta;\theta_0) = \phi(\theta)
\bigl[1 + O\bigl(e^{- \alpha t}\bigr)\bigr].
\end{equation}
Implicit in this formulation is the fact that
for every $t_0 >0$ there exists $C = C(t_0) < \infty$
such that if $t \geq t_0$,
$C^{-1} \phi(\theta) \leq\phi_t(\theta;\theta_0) \leq
C \phi(\theta)$.

If we apply this to (\ref{aug84}) with $X \equiv1$, we get
\[
{\mathbb P}_\theta\{T > t\} = c_* e^{(d-2)t} [\sin\theta]
^{4a-1} \bigl[1+O\bigl(e^{-\alpha
t}\bigr)\bigr],
\]
where
\[
c_* = \int_0^\pi[\sin\theta]^{1-4a}
\phi(\theta) \,d\theta= 2 C_{4a}^{-1}.
\]
In particular, we see that for $r\geq1/10$,
\[
{\mathbb P}_\theta\{T > \tau_r\} \asymp[\sin
\theta]^{4a-1} e^{(d-2)r},
\]
and if $r \geq3, 0\leq s \leq1/10$, and $T > r-2$,
conformal invariance implies that
\begin{eqnarray*}
{\mathbb P}_\theta\{ T > \tau_{r+s} \neq\varnothing\mid
\gamma_{r-2}\} & = & {\mathbb P}_{\theta_{r-2}} \bigl\{ \gamma\cap
g_{r-2}({\mathbb D}_{r+s}) \neq\varnothing\bigr\}
\\
& \asymp& [\sin\theta_{r-2}]^{4a-1}.
\end{eqnarray*}

We can also phrase this in terms of the quasi-stationary distribution
for $\theta_t$. Let $\psi(\theta) = \frac{1}2 \sin\theta$. Under the
measure ${\mathbb P}$, the random variable
$\theta_t 1\{T > t\}$ has an atom at $0$ and has a density
$\psi_t(\theta)$ for $0 < \theta< \pi$ satisfying
\[
{\mathbb P}\{T > t \} = \int_0^\pi
\psi_t(\theta) \,d \theta.
\]
The results of the previous paragraph show
that $\psi$ is a quasi-stationary
density in the sense that if $\phi_0 \equiv\psi$, then
\[
\psi_t(\theta) = e^{t(d-2)} \psi(\theta).
\]
Moreover, if $\psi_t(\theta;\theta_0)$ denotes the density assuming
initial condition $\theta_0$,
%
\begin{eqnarray}\label{station}
\psi_t(\theta;\theta_0) &=& {\mathbb
P}_{\theta_0}\{T > t\} \psi(\theta) \bigl[1 + O\bigl(e^{-t \alpha}\bigr)
\bigr]
\nonumber\\[-8pt]\\[-8pt]
&=& c_* e^{t(d-2)} \psi(\theta) \bigl[1 + O\bigl(e^{-t \alpha}
\bigr)\bigr].\nonumber
\end{eqnarray}
We write ${\mathbb P}_\psi$ for probabilities assuming the
initial density $\psi$. We can see
\[
{\mathbb P}_\psi\{T>r\}=e^{r(d-2)}.
\]

\begin{proposition} \label{jan26prop1}
There exists $0 < c_1 < \infty$ such that
\[
{\mathbb P}_\psi\{\tau_r < \infty\} = c_1
e^{r(d-2)} \bigl[ 1 + O\bigl(e^{-r}\bigr)\bigr].
\]
\end{proposition}

\begin{pf} If $r>0, u > 2$, then since $\tau_r > r-2$,
\[
{\mathbb P}_\psi\{\tau_{r+u} < \infty\} = {\mathbb
P}_\psi\{T > r \} {\mathbb P}_\psi\{\tau_{r+u} <
\infty\mid T > r\}.
\]
The conformal Markov property implies that
if $T > r$, then
\[
{\mathbb P}_\psi\{\tau_{r+u} < \infty\mid
\gamma_r\} = {\mathbb P}_{\theta_r} \bigl\{\gamma[0,\infty)
\cap g_r({\mathbb D}_{r+u}) \neq\varnothing\bigr\},
\]
where $\theta_r$ started according to $\psi$. By Lemma~\ref{growth},
there exists $u_0$ such that if
$u > u_0$, then on the event $T > r$, if $|z| = e^{-r-u}$,
\[
e^{-u} \exp\bigl\{-4 e^{-u} \bigr\} \leq\bigl|g_r(z)\bigr|
\leq e^{-u} \exp\bigl\{ 4 e^{-u} \bigr\}.
\]
Combining the last two expressions, we see that if $r > 0$
and $u > u_0$, then
if $T > r$,
%
\begin{eqnarray}\label{jan261}
{\mathbb P}_{\theta_r} \{\tau_{u + 4e^{-u}} < \infty\}
&\leq& {\mathbb P}_{\theta_r} \bigl\{\gamma[0,\infty) \cap g_r({
\mathbb D}_{r+u}) \neq\varnothing\bigr\}
\nonumber\\[-8pt]\\[-8pt]
&\leq& {\mathbb P}_{\theta_r}
\{\tau_{u - 4e^{-u}} < \infty\}.\nonumber
\end{eqnarray}
Since $\psi$ is the quasi-stationary density, the conditional
density on $\theta_r$ given $T > r$ is $\psi$. Therefore,
\[
e^{r(d-2)} {\mathbb P}_\psi\{\tau_{u + 4e^{-u}} < \infty\}
\leq{\mathbb P}_\psi\{\tau_{r+u} < \infty\} \leq
e^{r(d-2)} {\mathbb P}_\psi\{\tau_{u - 4e^{-u}} < \infty\}.
\]
If we replace $r$ with $s = r - 5e^{-u}$ and $u$ with
$v = u + 5e^{- u}$, we get for $u$ sufficiently large so
that $e^{-v} \geq(4/5) e^{-u}$,
\begin{eqnarray*}
{\mathbb P}_\psi\{\tau_{r+u} < \infty\} & = & {\mathbb
P}_\psi\{\tau_{s + v} < \infty\}
\\
& \leq& e^{s(d-2)} {\mathbb P}_\psi\{\tau_{v -4 e^{-v} }<
\infty\}
\\
& \leq& e^{s(d-2)} {\mathbb P}_\psi\{\tau_{u }<
\infty\}
\\
& = & e^{r(d-2)} {\mathbb P}_\psi\{\tau_{u }< \infty
\} \bigl[1 + O\bigl(e^{-u}\bigr)\bigr].
\end{eqnarray*}
We get a bound in the other direction by choosing $s = r+5e^{-u}$
and $v = r = e^{-5u}$. Hence,
\[
{\mathbb P}_\psi\{\tau_{r+u} < \infty\} = e^{r(d-2)}
{\mathbb P}_\psi\{\tau_u < \infty\} \bigl[ 1 + O
\bigl(e^{-u}\bigr)\bigr],
\]
where the error term is bounded uniformly independent of $r$.
If we define
$L_r = \log[e^{r(2-d)} {\mathbb P}_\psi\{\tau_{r} < \infty\}]$,
then the above expression can be written as
\[
\sup_{r \geq u} |L_r - L_u| = O
\bigl(e^{-u}\bigr),
\]
which\vspace*{1pt} implies that the limit $ L_\infty= \lim_{u \rightarrow\infty}
L_u \in(-\infty, \infty)$ exists, and $|L_u - L_\infty| = O(e^{-u})$.
The proposition follows with $c_1= e^{L_\infty}$.
\end{pf}

\begin{theorem} \label{lime}
There exists $0 < \hat c < \infty$ and $\beta> 0$, such
that
%
\begin{equation}
\label{jan265} {\mathbb P}_\theta\{\tau_r < \infty\} =
\hat c [\sin\theta]^{4a-1} e^{r(d-2)} \bigl[1+ O
\bigl(e^{-r \beta}\bigr)\bigr].
\end{equation}
\end{theorem}

\begin{pf}
As in (\ref{jan261}),
\[
{\mathbb P}_{\theta_r} \{\tau_{r+ 4e^{-r}} < \infty\}\leq{\mathbb
P}_{\theta_r} \bigl\{\gamma[0,\infty) \cap g_r({\mathbb
D}_{2r}) \neq\varnothing\bigr\} \leq{\mathbb P}_{\theta_r} \{
\tau_{r - 4e^{-r}} < \infty\}.
\]
By Proposition~\ref{jan26prop1}, if $\psi$ is the invariant distribution,
\[
{\mathbb P}_{\psi} \{\tau_{r \pm4e^{-r}} < \infty\} = c_1
e^{r(d-2)} \bigl[1 + O\bigl(e^{-r}\bigr)\bigr].
\]
Combining this with (\ref{station}), we see that
\begin{eqnarray*}
{\mathbb P}_{\theta} \{\tau_{2r} < \infty\} & = & {\mathbb
P}_{\theta} \{T > r \} {\mathbb P}_\theta\{\tau_{2r} <
\infty\mid T > r\}
\\
& = & c_1 c_* e^{2r(d-2)} \bigl[1 + O\bigl(e^{-\alpha r}
\bigr)\bigr].
\end{eqnarray*}\upqed
\end{pf}

With this theorem we could \textit{define} the chordal Green's function
on ${\mathbb D}$ by
\[
G_{\mathbb D}\bigl(0;1,e^{2i\theta}\bigr) = \hat c [\sin
\theta]^{4a-1},
\]
and define it for other simply connected domains by
\[
G_D(z;w_1,w_2) = \bigl|f'(z)\bigr|^{2-d}
G_{\mathbb D}\bigl(0;1,e^{2i\theta}\bigr),
\]
where $f\dvtx  D \rightarrow{\mathbb D}$ is a conformal transformation with
$f(z) = 0$, $f(w_1) = 1$, $f(w_2) = e^{2i\theta}$. In fact,
\[
G_D(z;w_1,w_2) = \hat c
\crad_D(z)^{d-2} S_D(z;w_1,w_2)^{4a-1}.
\]

\begin{proposition} \label{onesided}
Let $0 < \kappa< 8$.
There exists $c < \infty, \alpha> 0$ such that
the following is true. Suppose that $D$ is a simply connected domain and
$\gamma$ is a chordal $\SLE_\kappa$ path from $w_1$ to
$w_2$ in $D$. Suppose that $z \in D$, $R = \dist(z,\partial D)$ and
$G =
G_D(z;w_1,w_2)$.
Then if $e^{-r} \leq R/2$,
\[
\bigl\llvert G ^{-1} e^{r(2-d)} {\mathbb P} \bigl\{\dist(
\gamma,z) \leq e^{-r} \bigr\} - 1 \bigr\rrvert\leq c
\bigl[e^{- r}/R\bigr]^\alpha.
\]
In particular, there exists $c < \infty$ such
that if $0 < r < s$, then
%
\begin{eqnarray}\label{jan281}
&& \bigl\llvert{\mathbb P} \bigl\{\dist(\gamma,z) \leq
e^{-r} \bigr\} - e^{(s-r)(2-d)}{\mathbb P} \bigl\{\dist(\gamma,z)
\leq e^{-s} \bigr\} \bigr\rrvert
\nonumber\\[-8pt]\\[-8pt]
&&\qquad\leq c \bigl[e^{- r}/R\bigr]^{2-d+\alpha}.\nonumber
\end{eqnarray}
\end{proposition}

\begin{pf}
Without loss of generality, we assume $z=0$, and by scaling
we may assume that $R = 1$. Let
$F\dvtx D \rightarrow{\mathbb D}$ be the conformal transformation
with $F(0) = 0, F(w_1) = 1, F(w_2) = e^{2i\theta}$ where
$\sin\theta= S_D(z;w_1,w_2)$. The Schwarz
lemma and Koebe $(1/4)$-theorem imply that
$1/4 \leq|F'(0)| \leq1$. Note that $G
= \hat c |F'(0)|^{2-d} [\sin\theta]^{4a-1}$. Proposition
\ref{growth} implies that there exists universal $r_0$
such that if $r > r_0$,
\[
\bigl|F'(0)\bigr| |z| \exp\bigl\{-4|z|\bigr\} \leq\bigl|F(z)\bigr| \leq\bigl|F'(0)\bigr| |z|
\exp\bigl\{4 |z|\bigr\}.
\]
Therefore, by conformal invariance, if $q = -\log|F'(0)|$
\[
{\mathbb P}_{\theta} \{\tau_{r+ 4 e^{-r}+ q} < \infty\} \leq{\mathbb P}
\bigl\{\dist(\gamma,z) \leq e^{-r} \bigr\} \leq{\mathbb
P}_{\theta} \{\tau_{r- 4 e^{-r}+ q} < \infty\}.
\]
But (\ref{jan265}) tells us that
\begin{eqnarray*}
{\mathbb P}_{\theta} \{\tau_{r \pm4 e^{-r}+ q} < \infty\} & = & \hat c
[\sin
\theta]^{4a-1} e^{(r+q)(d-2)} \bigl[1 + O\bigl(e^{-\alpha r}\bigr)
\bigr]
\\
& = & G e^{(r+q)(d-2)} \bigl[1 + O\bigl(e^{-\alpha r}\bigr)\bigr].
\end{eqnarray*}\upqed
\end{pf}

With these results, we can
follow the proof in \cite{LW}, Section~3, which proves the
corresponding result
with distance replaced by conformal radius, to conclude~(\ref{twogreen}). We need
to replace
Lemma 2.16 of \cite{LW}, with the corresponding
result for the distance. The necessary lemma, written
in the notation of this paper, is the following.

\begin{lemma} \label{july}
There exist $\alpha> 0$, $c < \infty$ such that if
$0 < s < u < 1$ and $r \geq3$,
\[
{\mathbb P}_\theta\bigl\{\tau_r < \infty, \gamma(
\tau_{ur},\tau_r) \not\subset{\mathbb D}_{sr}
\bigr\} \leq c [\sin\theta]^{4a-1} e^{r(d-2)} e^{-\alpha tr},
\]
where $t = \min\{1-u, u-s\}$.
\end{lemma}

\begin{pf} A corresponding result was proved for two-sided
radial $\SLE_\kappa$ in \cite{Law3}. In particular, there exist
$c,\alpha$ such that
\[
{\mathbb P}_\theta^* \bigl\{\gamma(\tau_{ur},
\tau_{r}) \not\subset{\mathbb D}_{sr} \bigr\} \leq c
e^{\alpha( s-u)r}.
\]
In particular, since $\tau_r > r-2$,
\[
{\mathbb P}_\theta^* \bigl\{\gamma(\tau_{ur},r-2) \not
\subset{\mathbb D}_{sr} \bigr\} \leq c e^{\alpha( s-u)r}.
\]
Using the definition of the measure ${\mathbb P}_\theta^*$ we see
that this implies that
\begin{eqnarray*}
&& {\mathbb E}_\theta\bigl[[\sin\theta_{ {r-2}}]^{1-4a};
T > r-2, \gamma(\tau_{ur}, r-2) \not\subset{\mathbb D}_{sr}
\bigr]
\\
&&\qquad\leq c [\sin\theta]^{4a-1} e^{r(d-2)}
e^{\alpha( s-u)r}.
\end{eqnarray*}
However, if $T > r-2$,
${\mathbb P}\{\tau_r < \infty\mid\gamma_{r-2}\}
\asymp[\sin\theta_{ {r-2}}]^{4a-1}$. Hence,
%
\begin{equation}
\label{jan291} {\mathbb P}_\theta\bigl\{ \tau_{r} <
\infty, \gamma(\tau_{ur},r-2) \not\subset{\mathbb D}_{sr}
\bigr\} \leq c [\sin\theta]^{4a-1} e^{r(d-2)} e^{\alpha( s-u)r}.
\end{equation}
On the event $E:= \{ T > r-2, \gamma(\tau_{ur},{r-2})
\subset{\mathbb D}_{sr} \} $, topological considerations
(see \cite{Law3}, Lemma 2.3) imply that there is a unique subarc
$l$ of $\partial{\mathbb D}_{ur} \cap D_{{r-2}}$ such that
removal of $l$ disconnects $0$ from $\partial{\mathbb D}_{sr}$
in $D_{r-2}$. The point $\gamma(r-2)$ may be in
$l$ or in
either of the connected components of $D_{r-2} \setminus l$.
In any of these cases,
if $\sigma= \inf\{t \geq r-2\dvtx  \gamma(t)
\in\bar l\}$, then
the event $E \cap\{\tau_r < \infty,
\gamma(\tau_{ur}, \tau_r) \not\subset{\mathbb D}_{sr}\}$
is contained in the event $E \cap\{\sigma< \tau_r < \infty\}$.
As in (\ref{inparticular}), on the event $E \cap\{\sigma< \infty,
\sigma< \tau_r\}$,
\[
{\mathbb P}\{\tau_r < \infty\mid\gamma_\sigma\} \leq c
e^{\alpha(u-1)r}.
\]
Hence,
%
\begin{eqnarray}\label{jan292}
&& {\mathbb P}_\theta\bigl\{ \tau_r < \infty,
\gamma(\tau_{ur},\tau_{r}) \not\subset{\mathbb
D}_{sr}, \gamma(\tau_{ur},r-2) \subset{\mathbb
D}_{sr} \bigr\} \nonumber
\\
&&\qquad \leq {\mathbb P}_\theta(E) {\mathbb P}_\theta\bigl\{
\tau_r < \infty,\gamma(\tau_{ur},\tau_{r})
\not\subset{\mathbb D}_{sr} \mid E\bigr\}
\nonumber\\[-8pt]\\[-8pt]
&&\qquad\leq{\mathbb P}_\theta\{T > r-2\} {\mathbb P}_\theta\bigl
\{\tau_r < \infty,\gamma(\tau_{ur},\tau_{r})
\not\subset{\mathbb D}_{sr} \mid E\bigr\}\nonumber
\\
&&\qquad\leq c [\sin\theta]^{4a-1} e^{r(d-2)} e^{\alpha
r(u-1)}.\nonumber
\end{eqnarray}
The lemma follows from (\ref{jan291}) and (\ref{jan292}).
\end{pf}

The proof of (\ref{twogreen}) follows
that in \cite{LW}, Section~3. We will not give all the
details, but we sketch the argument using the notation
of this paper. We need to prove the existence of the
limit
\[
G(z,w) = \lim_{r,s \rightarrow\infty} e^{r(2-d)} e^{s(2-d)} {
\mathbb P}\bigl\{\tau_r(z), \tau_s(w) < \infty\bigr\}.
\]

\begin{proposition} For every $z,w \in{\mathbb H}$,
\[
\lim_{r,s \rightarrow\infty} e^{r(2-d)} e^{s(2-d)} {\mathbb P}\bigl
\{\tau_r(z) < \tau_s(w) < \infty\bigr\} = G(z) {\mathbb
E}^* \bigl[G_{H_T}(w;z,\infty) \bigr],
\]
where ${\mathbb E}^*$ denotes expectation with respect
to two-sided radial SLE to $z$.
\end{proposition}

\begin{pf*}{Proof (sketch)}
Let $\tau_r = \tau_r(z) $.
Arguing as in (\ref{aug205}), we see that
\[
\lim_{r,s \rightarrow\infty} e^{r(2-d)} e^{s(2-d)} {\mathbb P}\bigl
\{\tau_{s/2}(w) < \tau_r < \tau_s(w) < \infty
\bigr\} = 0.
\]
Also, by (\ref{aug205}) and Proposition~\ref{onesided},
\begin{eqnarray*}
\lim_{r,s \rightarrow\infty} e^{r(2-d)} {\mathbb P}\bigl\{
\tau_r < \tau_{s/2}(w) \bigr\} = \lim_{r \rightarrow\infty}
e^{r(2-d)} {\mathbb P}\{ \tau_r < \infty\}
&=& G(z).
\end{eqnarray*}
By Proposition~\ref{onesided}, there exists $\alpha$
such that if $\tau< \tau_{s/2}(w)$,
\[
{\mathbb P}\bigl\{ \tau_s(w) < \infty\mid\gamma_{\tau_r}\bigr
\} = e^{s(d-2)} G_{ H_{\tau_r}}\bigl(w;\gamma(\tau_r),
\infty\bigr) \bigl[1 + O\bigl(e^{-\alpha s}\bigr)\bigr].
\]
Therefore,
\begin{eqnarray*}
&& \lim_{r,s \rightarrow\infty} e^{r(2-d)} e^{s(2-d)}
G(z)^{-1} {\mathbb P}\bigl\{\tau_r <
\tau_s(w) < \infty\bigr\}
\\
&&\qquad=
\lim_{r,s \rightarrow\infty} {\mathbb E} \bigl[ G_{ H_{\tau_r}}
\bigl(w;\gamma(\tau_r),\infty\bigr) 1\bigl\{\tau_r <
\tau_{s/2}(w)\bigr\} \mid\tau_r < \infty\bigr].
\end{eqnarray*}
Hence, we need to show that the right-hand side
equals ${\mathbb E}^*[G_{H_T}(w;z,\infty)]$.

We assume that the curve has the radial parameterization
heading to $z$.
We use Lemma~\ref{july} to see that as $r,s
\rightarrow\infty$,
\begin{eqnarray*}
&& {\mathbb E} \bigl[ G_{ H_{\tau_r}}\bigl(w;\gamma(\tau
_r),\infty\bigr) 1\bigl\{\tau_r <
\tau_{s/2}(w)\bigr\} \mid\tau_r < \infty\bigr]
\\
&&\qquad\sim{\mathbb E} \bigl[ G_{ H_{r/2}}\bigl(w;\gamma(r/2),\infty
\bigr) 1
\bigl\{\tau_r < \tau_{s/2}(w)\bigr\} \mid
\tau_r < \infty\bigr]
\\
&&\qquad\sim{\mathbb E} \bigl[ G_{ H_{r/2}}\bigl(w;\gamma(r/2),\infty
\bigr) {
\mathbb P}\bigl\{\tau_r < \tau_{s/2}(w) \mid
\gamma_{r/2}\bigr\} \mid\tau_r < \infty\bigr].
\end{eqnarray*}
We now use Proposition~\ref{onesided} to see that
the weighting by ${\mathbb P}\{\tau_r < \tau_{s/2}(w) \mid\gamma
_{r/2}\}$
is the same up to small error as weighting
by $G_{H_{s/2}}(z;\gamma(s/2),\infty)$ which is the
weighting which defines two-sided $\SLE$ going to $z$.
The arguments for justifying this are the same whether
one uses conformal radius or $\tau_r$ as the stopping
time, so the proof in \cite{LW}, Section~3, works here.
\end{pf*}

\begin{remark*}
The same method shows that we can define $n$-point Green's function and
we expect that
\[
{\mathbb E}\bigl[\Theta(D)^n\bigr]=\int_{D^n}
G(z_1,\ldots,z_n) \,dA(z_1)\cdots dA(z_n).
\]
At the moment, we cannot prove it because we have no upper bound for
$G(z_1,\ldots,z_n)$.
\end{remark*}


\subsection{Proof of (\texorpdfstring{\protect\ref{mainesalt}}{3.6})}

We will now prove (\ref{mainesalt}) for an appropriate $0 < u \leq1/4$
that we will define below. We assume that
$\Im(z),\Im(w) \geq1$ and $|z-w| > e^{-r/4}$. It suffices to
prove the result for $r > 4$, and hence $|z-w| > e^{-(r-2)/2}$.

Let $0 < q < 1/8$ be a parameter that we will
choose later. Let $\tau_{r} = \tau_{r}(z),
H = H_{\tau_{r}}(z), l_{3/4} = l_{3/4}(r,z),
\lambda= \lambda(r,z,3/4),
{\mathcal B}_u= {\mathcal B}_u(r,z), V_u = V_u(r,z)$ be
as in Section~\ref{twopointsec}.
Recall that we are assuming that
$|z-w| > e^{-r/2}$ and hence
$w \notin{\mathcal B}_{1/2}$.
Let $ {\mathcal I}_r(z,w)$
be the indicator function of the event that $\tau_{r} < \tau_{
q r}(w)$ and $w$ is
in the unbounded component of $H
\setminus l_{3/4}$ and let ${\mathcal J}_r(z,w)$
be the indicator function of the event that $w$ is
in the bounded component of $H
\setminus l_{3/4}$.

We consider two cases. First, suppose that
${\mathcal J}_r(z,w) = 1$. Then $w,z$ are both in the bounded
component of $H \setminus l_{3/4}$. Since $w \notin
{\mathcal B}_{1/2}$, there is
a unique subarc $l'$ of $\partial V_{3/4} \cap H$ such that
$z,w$ are in different components of $H \setminus l'$.
Since $w$ is in the bounded component of $H \setminus l_{3/4}$,
$l' \ne l_{3/4}$. In particular, $z$ is in the unbounded
component of $H \setminus l'$.
The Beurling estimate implies that the probability
that a Brownian motion starting at $w$ reaches
$l'$ without leaving $H$ is bounded above by $c e^{-r/8}$.
Hence, $S_{\tau_{r}} (w)\leq c e^{-r/8}$ and, therefore,
\begin{eqnarray*}
{\mathbb P}\bigl\{\tau_{r}(w) < \infty\mid\gamma_{\tau_{r} }
\bigr\} & \leq& c G_{H}\bigl(w;\gamma(\tau_{r}),\infty
\bigr) e^{r(d-2)}
\\
& \leq& c S_{\tau_{r} }(w)^{4a-1} \dist(w, \gamma_{\tau_{r }})^{d-2}
e^{r(d-2)}
\\
& \leq& c e^{-p r} \dist(w, \gamma_{\tau_{r} })^{d-2}
e^{r(d-2)},
\end{eqnarray*}
where $p = (4a-1)/8 > 0$.
We know from (\ref{aug205}) that
\[
{\mathbb P}\bigl\{\dist(w, \gamma_{\tau_{r}}) \leq e^{s},
\tau_{r}< \infty\bigr\} \leq c |w-z|^{d-2} e^{r(d-2)}
e^{s(d-2)}.
\]
By summing over positive integers $s \leq r$, we get
\[
{\mathbb P}\bigl\{\tau_{r} < \tau_{r}(w) < \infty, {
\mathcal J}_r(z,w) =1 \bigr\} \leq c r |w-z|^{d-2}
e^{2r(d-2)} e^{-p r}.
\]
In particular, if $|z-w| \geq e^{-ur}$,
where $u = p/[3(2-d)]$,
\[
{\mathbb P}\bigl\{\tau_{r}(z) < \tau_{r}(w) < \infty, {
\mathcal J}_r(z,w) =1 \bigr\} \leq c e^{2r(d-2)}
e^{- p r/2},
\]
which implies that if $|z-w| \geq e^{-ur}$,
\[
{\mathbb E} \bigl[ Q_r(z) Q_r(w) {\mathcal
J}_r(z,w) \bigr] \leq c e^{-p r/2}.
\]

For the remainder, we will assume that
${\mathcal I}_r(z,w) = 1$. Let $\sigma= \sigma_{3/4}(r-2,z)$
as in Section~\ref{twopointsec}. Let $\widetilde Q_r(z)$
be the analogue of $Q_r(z)$ for the curve
stopped at time~$\sigma$,
\[
\widetilde Q_r(z) = {e^{r(2-d)} } \bigl[ 1 \bigl\{
\tau_{r }(z) < \sigma\bigr\} - e^{\delta(2-d)} 1 \bigl\{
\tau_{r + \delta}(z) < \sigma\bigr\} \bigr].
\]
%
To establish our estimate,
we will show that
%
\begin{equation}
\label{aug203} \bigl\llvert{\mathbb E} \bigl[ \widetilde Q_r(z)
Q_r(w) {\mathcal I}_r(z,w) \bigr] \bigr\rrvert\leq c
e^{-\beta r},
\end{equation}
and
%
\begin{equation}
\label{aug204} {\mathbb E} \bigl[ \bigl|Q_r(z) - \widetilde Q_r(z)\bigr| \bigl|Q_r(w)\bigr| {\mathcal I}_r(z,w)
\bigr] \leq c e^{-\beta r},
\end{equation}
which together imply that
\[
\bigl\llvert{\mathbb E} \bigl[ Q_r(z) Q_r(w) {
\mathcal I}_r(z,w) \bigr] \bigr\rrvert\leq c e^{-\beta r}.
\]

To prove (\ref{aug203}), note that since
$\widetilde Q_r(z) {\mathcal I}_r(z,w)$ is $\gamma_\sigma$
measurable,
\[
{\mathbb E} \bigl[ \widetilde Q_r(z) Q_r(w) {\mathcal
I}_r(z,w) \bigr] = {\mathbb E} \bigl[ \widetilde Q_r(z) {
\mathcal I}_r(z,w) {\mathbb E}\bigl(Q_r(w) \mid
\gamma_\sigma\bigr) \bigr],
\]
and hence,
\[
\bigl\llvert{\mathbb E} \bigl[ \widetilde Q_r(z) Q_r(w)
{\mathcal I}_r(z,w) \bigr]\bigr\rrvert\leq{\mathbb E} \bigl[ \bigl|
\widetilde Q_r(z)\bigr| {\mathcal I}_r(z,w) \bigl|{\mathbb E}
\bigl(Q_r(w) \mid\gamma_\sigma\bigr)\bigr| \bigr].
\]
We appeal to (\ref{jan281}) to see that
\begin{eqnarray*}
\bigl|{\mathbb E}\bigl(Q_r(w) \mid\gamma_\sigma\bigr)\bigr| & \leq&
c e^{(2-d)r} \bigl[ e^{-r}/ \dist(w,\partial H_\sigma)
\bigr]^{2-d+ \alpha}
\\
& \leq& c \exp\bigl\{\bigl[(2-d) + (q-1) (2-d +\alpha)\bigr]r\bigr\}.
\end{eqnarray*}
%
In particular, if $q$ is chosen sufficiently small
so that $q(2-d) \leq\alpha(1-q)/2$,
\[
\bigl|{\mathbb E}\bigl(Q_r(w) \mid\gamma_\sigma\bigr)\bigr| \leq c
e^{-\alpha r/2},
\]
and hence
\begin{eqnarray*}
\bigl\llvert{\mathbb E} \bigl[ \widetilde Q_r(z) Q_r(w)
{\mathcal I}_r(z,w) \bigr]\bigr\rrvert& \leq& c e^{-r \alpha/2}
{\mathbb E} \bigl[ \bigl|\widetilde Q_r(z)\bigr| {\mathcal I}_r(z,w)
\bigr]
\\
& \leq& c e^{-r \alpha/2} e^{r(2-d)} {\mathbb P}\bigl[{\mathcal
I}_r(z,w)\bigr] \leq c e^{-r \alpha/2}.
\end{eqnarray*}

For (\ref{aug204}), we observe that
if $Q_r(z) \neq\widetilde Q_r(z)$, then
$\dist(z,\gamma) < \dist(z,\gamma_\sigma)$.
In other words,
\begin{eqnarray*}
&& {\mathbb E} \bigl[ \bigl|Q_r(z) - \widetilde Q_r(z)\bigr|
\bigl|Q_r(w)\bigr| {\mathcal I}_r(z,w) \bigr]
\\
&&\qquad\leq
c e^{2r(2-d)} {\mathbb P}\bigl\{{\mathcal I}_r(z,w)
= 1, \dist(z,\gamma) < \dist(z,\gamma_\sigma), \tau_{r}(w)
< \infty\bigr\}.
\end{eqnarray*}
We know that ${\mathbb P}\{{\mathcal I}_r(z,w) = 1\} \leq c
e^{r(d-2)}$. Hence, it suffices to show
that we can find $q, \beta> 0$ and $c < \infty$
such that
that on the event $ \{{\mathcal I}_r(z,w) = 1\} $,
\[
{\mathbb P}\bigl\{ \rho< \infty, \tau_r(w) < \infty\mid
\gamma_\sigma\bigr\} \leq c e^{r(2-d)} e^{-r \beta},
\]
where
$\rho= \inf\{t \geq\sigma\dvtx  |\gamma(t) - z| =
\dist(z,\gamma_\sigma)\}$. For every
integer $k$ with $qr + 1 < k < r-1$, we consider the event
\[
E_k = \bigl\{ \tau_k(w) < \rho< \tau_{k+1}(w)
< \tau_r(w)< \infty\bigr\}.
\]
We claim that there exists $c,\alpha$ such that
on the event $ \{{\mathcal I}_r(z,w) = 1\} $
%
\begin{equation}
\label{jan2810} {\mathbb P}(E_k \mid\gamma_\sigma) \leq c
e^{(2-d)(q-1) r } e^{-\alpha r}.
\end{equation}
Indeed, recall that on the
event on the event $ \{{\mathcal I}_r(z,w) = 1\} $,
$\dist(w,\partial H_\sigma) \geq e^{-qr}$.
If we use $\widetilde{\mathbb P}$ to denote conditional probabilities
given $\gamma_\sigma$, then
\begin{eqnarray*}
\widetilde{\mathbb P}\bigl\{\tau_k(w) < \infty\bigr\} &\leq& c
e^{(d-2)(k - qr)},
\\
\widetilde{\mathbb P}\bigl\{ \rho< \infty\mid\tau_k(w) < \infty\bigr\}
&\leq& c e^{-\alpha r},
\\
\widetilde{\mathbb P}\bigl\{ \tau_r(w)< \infty\mid
\tau_k(w) < \rho< \tau_{k+1}(w) < \infty\bigr\} &\leq& c
e^{(d-2)(r-k)}.
\end{eqnarray*}
By summing (\ref{jan2810}) over $k$, and then
choosing $q$ sufficiently small, we see that
\[
\widetilde{\mathbb P}\bigl\{ \rho< \infty, \tau_r(w) < \infty\bigr\}
\leq c r e^{(2-d)(q-1) r } e^{-\alpha r} \leq c e^{r(d-2)}
e^{-\alpha r/2}.
\]

At the end, we want to thank the referees for very careful reading of
the earlier draft of the paper and making many useful comments on that.
%
%



%

\printaddresses

\end{document}